%% file: main.tex
\documentclass{article}
\usepackage{graphicx} 
\usepackage{amsmath}
\usepackage{amsfonts}
\usepackage{amssymb}
\usepackage{amsthm}
\usepackage{array}
\usepackage{latexsym,amsmath}
\usepackage{graphicx}
\usepackage{afterpage}
\usepackage[utf8]{inputenc}
\usepackage[dvipsnames]{xcolor}
\usepackage{tikz}

\usepackage{tikz-cd}
\usepackage{float}
\usepackage{caption}
\usepackage{lipsum}
\usepackage{enumitem}
\usepackage[english]{babel}
\usepackage[nottoc]{tocbibind}
\usetikzlibrary{decorations.pathreplacing,calligraphy}
\usetikzlibrary{calc}
\usepackage{tikzit}

\usepackage{bbm}
\usepackage{leftindex}
\usepackage{tensor}
\usepackage{hyperref}
\usepackage[margin=1in]{geometry}

\input{sample.tikzstyles}

\newcommand{\Z}{\mathbb{Z}}
\newcommand{\al}{\alpha}

\newcommand{\cL}{\mathcal{L}}

\newcommand{\cM}{\mathcal{M}}

\newcommand{\lgl}{\langle}
\newcommand{\rgl}{\rangle}

\newcommand{\dl}{\delta}

\newcommand{\un}{\underline}
\newcommand{\jw}{\prescript{J}{}{W}}
\newcommand{\ibj}{\tensor*[^I]{\mathcal{B}}{^J}}
\newcommand{\jbk}{\tensor*[^J]{\mathcal{B}}{^K}}
\newcommand{\ibk}{\tensor*[^I]{\mathcal{B}}{^K}}
\newcommand{\isj}{\tensor*[^I]{\mathcal{S}}{^J}}
\newcommand{\jsk}{\tensor*[^J]{\mathcal{S}}{^K}}
\newcommand{\isk}{\tensor*[^I]{\mathcal{S}}{^K}}
\newcommand{\ol}{\overline}
\newcommand{\jbs}{J\mathbb{BS}\text{Bim}}
\newcommand{\tA}{\tilde{\mathcal{A}}}

\theoremstyle{plain}
\newtheorem{thm}{Theorem}[section]
\newtheorem{lem}[thm]{Lemma}
\newtheorem{prop}[thm]{Proposition}
\newtheorem{cor}[thm]{Corollary}

\theoremstyle{definition}
\newtheorem{defn}[thm]{Definition}
\newtheorem{conj}[thm]{Conjecture}
\newtheorem{exmp}[thm]{Example}
\newtheorem{rem}[thm]{Remark}
\title{The Diagrammatic Spherical Category}
\author{Tasman Fell \\ University of Sydney}
\date{February 2025}

\begin{document}

\maketitle

\begin{abstract}
    We construct a diagrammatic categorification of the spherical module over the Hecke algebra. We establish a basis for the morphism spaces of this category, and prove that it is equivalent to an existing algebraic spherical category. 
\end{abstract}

\section{Introduction}
Let $G$ be a reductive algebraic group over an algebraically closed field of characteristic $p$, with Weyl group $W$ generated by simple reflections $S$. A natural goal is to find the characters of simple $G$-modules. Lusztig gave a conjectural answer in 1980 \cite{lusztig1980}, known as Lusztig's conjecture, which expressed simple characters in terms of affine Kazhdan-Lusztig polynomials. Lusztig's conjecture has been shown to hold for large $p$ --- see work of Andersen, Jantzen and Soergel \cite{andersen1994}, Kashiwara and Tanisaki \cite{kashiwara1995}, \cite{kashiwara1996}, Kazhdan and Lusztig \cite{kazhdan1993}, \cite{kazhdan1994iii}, \cite{kazhdan1994iv}, and Lusztig \cite{lusztig1994}. However, Williamson has provided an infinite family of counter-examples to the conjecture \cite{williamson2017}. Further work has demonstrated that the correct formula requires replacing Kazhdan-Lusztig polynomials with $p$-Kazhdan-Lusztig polynomials. In particular, Riche and Williamson have demonstrated that simple characters can be computed using the $p$-canonical basis of the \textit{spherical module} $M(J)$ \cite{RicheWilliamson2021}, a module over the Hecke algebra which depends on a choice $J \subset S$. However, while Kazhdan-Lusztig polynomials can be computed using a recursive formula, there is no known recursive formula for computing $p$-Kazhdan-Lusztig polynomials. Instead, they are found by computing object decompositions in Hecke categories defined over prime characteristic.

The diagrammatic Hecke category $Kar(\mathcal{D})$, introduced by Elias and Williamson \cite{SoergelCalculus}, is well-suited to such computations, since it is well-behaved in characteristic $p$, and is easier to compute with than algebraic Hecke categories such as Soergel bimodules. In order to use Riche and Williamson's simple character formula, we need to compute object decompositions not in a Hecke category, but in a \textit{spherical category}, that is, a categorification of the spherical module $M(J)$. For this purpose, it will be useful to have a diagrammatic spherical category, analogous to the diagrammatic Hecke category. The purpose of this paper is to construct such a category. More precisely, we construct a diagrammatic category $\mathcal{M}_{BS}(J)$, whose Karoubi envelope $\mathcal{M}(J) = Kar(\mathcal{M}_{BS}(J))$ is a spherical category. This was done in type A by Elias \cite{Thicker}, and this paper extends his work to all types. 

The paper presents three main results. The first main result is a basis for the morphism spaces of $\mathcal{M}_{BS}(J)$, called the \textit{double-leaves basis}. To be precise, the objects of $\mathcal{M}_{BS}(J)$ are indexed by expressions $\un{x}$ in the Coxeter system. Given two objects $\un{x}, \un{y} \in \mathcal{M}_{BS}(J)$, we present an algorithm which constructs a collection $\mathcal{SDL}_{\un{x}, \un{y}}$ of \textit{double-leaf maps} in Hom$_{\mathcal{M}_{BS}(J)}(\un{x}, \un{y})$. Let $R := \text{Sym}(\mathfrak{h}^*)$, where $\mathfrak{h}$ is a realization\footnote{Subject to certain conditions; see Section \ref{realizations}.} of $(W, S)$. Then we have the following.
\begin{thm} \label{result1}
    The collection $\mathcal{SDL}_{\un{x}, \un{y}}$ is a basis of Hom$_{\mathcal{M}_{BS}(J)}(\un{x}, \un{y})$ as a right $R$-module. 
\end{thm}
\noindent The name ``double-leaves'' follows the basis of the Hecke category of the same name, constructed by Libedinsky \cite{double-leaves} in the algebraic setting, and by Elias and Williamson \cite{SoergelCalculus} in the diagrammatic setting.

Using this basis, we prove our second main result, which is that our diagrammatic category $\mathcal{M}(J)$ does indeed categorify the spherical module.

\begin{thm}
    We have an isomorphism of $H$-modules \[ch: [\mathcal{M}(J)] \xrightarrow{\sim} M(J),\] where $[\mathcal{M}(J)]$ denotes the split Grothendieck group.
\end{thm}

The third main result of this paper is that the diagrammatic spherical category is equivalent to the \textit{algebraic spherical category}\footnote{This final result requires a strong assumption on our realization, namely that it is reflection-faithful.}. To explain what this algebraic spherical category is, it is necessary to introduce \textit{singular Soergel bimodules}. 

\subsection{Singular Soergel bimodules} Singular Soergel bimodules were studied by Williamson \cite{Singular} in his thesis, following earlier work in this direction by Soergel, Stroppel and others. The category $\mathcal{S}\mathbb{S}$Bim of singular Soergel bimodules is a 2-category associated to a Coxeter system $(W, S)$. To define it, we first construct the 2-category $\mathcal{S}\mathbb{BS}$Bim of \textit{singular Bott-Samelson bimodules}. Its objects are finitary\footnote{I.e. the subgroup $W_J \leq W$ generated by $J$ is finite.} subsets $J \subset S$. The 1-morphisms in Hom$_{\mathcal{S}\mathbb{S}\text{Bim}}(I, J)$ are a certain class of $(R^J, R^I)$-bimodules, for invariant rings $R^J, \; R^I$\footnote{See Chapter \ref{Singular Soergel Bimodules and Standard Bimodules}.}. The 2-morphisms are bimodule morphisms. Then, the category $\mathcal{S}\mathbb{S}$Bim has the same objects as $\mathcal{S}\mathbb{BS}$Bim, and we set $$\text{Hom}_{\mathcal{S}\mathbb{S}\text{Bim}}(I, J) := Kar(\text{Hom}_{\mathcal{S}\mathbb{BS}\text{Bim}}(I, J)),$$ where `$Kar$' denotes the Karoubian closure.

Just as the category $\mathbb{S}$Bim categorifies the Hecke algebra, the 2-category $\mathcal{S}\mathbb{S}$Bim categorifies the \textit{Schur algebroid $\mathcal{S}$} (see Section \ref{The Schur algebroid}), which can be viewed as a 1-category. The fact that $\mathbb{S}$Bim categorifies the Hecke algebra $H$ can be viewed as a special case of this higher categorification. That is, there is a Hom set Hom$_{\mathcal{S}}(\varnothing, \varnothing) \cong H$ in $\mathcal{S}$, with corresponding Hom category Hom$_{\mathcal{S}\mathbb{S}\text{Bim}}(\varnothing, \varnothing) \cong \mathbb{S}$Bim in $\mathcal{S}\mathbb{S}$Bim.

Similarly, there is a Hom set Hom$_{\mathcal{S}}(\varnothing, J)$ in $\mathcal{S}$ which is isomorphic to the spherical module $M(J)$. Therefore, the corresponding Hom category Hom$_{\mathcal{S}\mathbb{S}\text{Bim}}(\varnothing, J)$ in $\mathcal{S}\mathbb{S}$Bim categorifies $M(J)$, and is thus a spherical category. We call Hom$_{\mathcal{S}\mathbb{S}\text{Bim}}(\varnothing, J)$ the \textit{algebraic spherical category}. The third main result of this paper is that our diagrammatic and algebraic spherical categories are equivalent as module categories. More precisely, the diagrammatic spherical category $\mathcal{M}(J)$ is a right module category over the diagrammatic Hecke category $Kar(\mathcal{D})$. Similarly, the algebraic spherical category Hom$_{\mathcal{S}\mathbb{S}\text{Bim}}(\varnothing, J)$ is a right module category over Hom$_{\mathcal{S}\mathbb{S}\text{Bim}}(\varnothing, \varnothing) \cong \mathbb{S}$Bim. Work of Elias and Williamson \cite{SoergelCalculus} showed that we have an equivalence $Kar(\mathcal{D}) \cong \mathbb{S}$Bim. Then we have the following:

\begin{thm} \label{result2}
    Suppose our realization $\mathfrak{h}$ is reflection faithful (see Definition \ref{reflection faithful}). Then we have an equivalence of categories $\mathcal{M}_{BS}(J) \xrightarrow{\sim} \text{Hom}_{\mathcal{S}\mathbb{BS}\text{Bim}}(\varnothing, J)$. After taking Karoubian closures, this gives us an equivalence between the diagrammatic and algebraic spherical categories, $\mathcal{M}(J)$ and Hom$_{\mathcal{S}\mathbb{S}\text{Bim}}(\varnothing, J)$, as module categories over $Kar(\mathcal{D})$ and $\mathbb{S}$Bim respectively. That is, the following diagram commutes.
    
\begin{center}
    \input{fig136.tikz}
\end{center}
\end{thm}

There has been recent major progress by Elias, Ko, Libedinsky and Patimo in constructing a diagrammatic category to model $\mathcal{S}\mathbb{BS}$Bim. In \cite{SingularLightLeaves} the authors define a diagrammatic 2-category $\mathbf{Frob}$ intended to model $\mathcal{S}\mathbb{BS}$Bim.  In particular, they define an evaluation functor $\mathcal{F} :\mathbf{Frob} \rightarrow \mathcal{S}\mathbb{BS}$Bim. Further, they construct a basis for the 2-morphism spaces of $\mathcal{S}\mathbb{BS}$Bim. This basis is the image of a collection of double-leaf morphisms in $\mathbf{Frob}$, under the functor $\mathcal{F}: \mathbf{Frob} \rightarrow \mathcal{S}\mathbb{BS}$Bim.

When we restrict this result to the Hom category Hom$_{\mathcal{S}\mathbb{BS}\text{Bim}}(\varnothing, J)$, we end up with a double-leaves basis for Hom$_{\mathcal{S}\mathbb{BS}\text{Bim}}(\varnothing, J)$. This is similar to the main results of this paper. In particular, combining Theorems \ref{result1} and \ref{result2} also gives us a double-leaves basis for Hom$_{\mathcal{S}\mathbb{BS}\text{Bim}}(\varnothing, J)$. The addition made in this thesis is that this is also a basis of the diagrammatic category $\mathcal{M}_{BS}(J)$, not just the algebraic category Hom$_{\mathcal{S}\mathbb{BS}\text{Bim}}(\varnothing, J)$. It is this fact which is necessary to prove the equivalence in Theorem \ref{result2}.

\subsection{Structure of this paper}
We now describe the structure of the paper. Chapters 2 and 3 introduce the necessary background material.

\begin{itemize}
    \item Chapter \ref{Coxeter Systems and the Hecke Algebra}: We describe Coxeter systems and the Hecke algebra, and introduce the spherical and anti-spherical modules. We also define the Schur algebroid, where the spherical module appears as a Hom set.
    
    \item Chapter \ref{Singular Soergel Bimodules and Standard Bimodules}: We define the 2-category of singular Soergel bimodules and the algebraic spherical category. We also introduce the category of standard bimodules and discuss localization, which provides an essential tool for proving linear independence.
    
    \item Chapter \ref{Spherical Diagrammatics}: We construct the diagrammatic category $\mathcal{M}_{BS}(J)$. We then introduce spherical light-leaves and define the double-leaves basis for the morphism spaces of $\mathcal{M}_{BS}(J)$.
    
    \item Chapter \ref{Standard Diagrammatics}: We construct the category of \textit{spherical standard diagrammatics} $\mathcal{M}^{std}(J)$ and its localization. We establish equivalences between the localized diagrammatic categories and the algebraic categories of standard bimodules.
    
    \item Chapter \ref{Linear Independence}: We prove that the spherical double-leaves are linearly independent.

    \item Chapter \ref{Spanning}: We prove that the spherical double-leaves span morphism spaces in $\mathcal{M}_{BS}$ and thus form a basis.
    
    \item Chapter \ref{Consequences}: We discuss three consequences of the main theorem: We show that $\mathcal{M}(J)$ categorifies $M(J)$. We prove that $\mathcal{M}(J)$ is equivalent to Abe's category of one-sided Soergel bimodules. We prove that when $\mathfrak{h}$ is reflection-faithful, $\mathcal{M}(J)$ is equivalent to the algebraic spherical category $Kar(J\mathbb{BS}Bim)$ as module categories.

\section*{Acknowledgements}
The results of this paper are those of my master's thesis, completed under Anna Romanov, with secondary supervisors Geordie Williamson and Jie Du. I am very grateful to them for all their support, assistance and feedback. The seed of the spanning proof - the idea to turn spherical light-leaves into non-spherical light-leaves - came from Geordie Williamson. I would also like to thank Joe Baine for patiently answering my endless questions throughout the writing of my thesis. Thank you to Simon Riche for providing guidance and suggestions regarding the technical conditions on the realization, in particular pointing out that my faithfulness assumption was not necessary for many of the results.
\end{itemize}
\section{Coxeter Systems and the Hecke Algebra} \label{Coxeter Systems and the Hecke Algebra}

We follow the notation of \cite{textbook}. Let $(W, S)$ be a Coxeter system with length function $l: W \to \mathbb{N}$ and Bruhat order $\le$.

\begin{defn} \label{subexpression}
    Given an expression $\un{w} = (s_1, ... ,s_k)$, a \textit{subexpression} $\un{e} \subset \un{w}$ is a binary sequence $(e_1, ..., e_k)$ for $e_i \in\{0, 1\}$. We denote the element obtained by deleting the terms where $e_i=0$ as $\un{w}^{\un{e}} := s_1^{e_1} \cdots s_k^{e_k}$.
\end{defn}
\subsection{The Hecke algebra and Kazhdan-Lusztig basis}

Given a Coxeter system $(W, S)$, the \textit{Hecke algebra} $H$ is the $\mathbb{Z}[v, v^{-1}]$-algebra with generators $\delta_s$ for $s \in S$ subject to the braid relations and the quadratic relation $\dl_s^2 = 1 + (v^{-1} - v)\dl_s$.
The set $\{\dl_w \mid w \in W \}$ forms the \textit{standard basis} of $H$ over $\mathbb{Z}[v, v^{-1}]$, where $\dl_w := \dl_{s_1}\cdots\dl_{s_k}$ for $(s_1, ..., s_k)$ a reduced expression of $w$.

The \textit{Kazhdan-Lusztig involution} (or bar involution) is the ring automorphism $H \to H, h \mapsto \overline{h}$ determined by $\overline{v} = v^{-1}$ and $\overline{\dl_s} = \dl_s^{-1}$.
\begin{thm}[\cite{KL}]
    For each $x \in W$, there exists a unique element $b_x \in H$ such that:
    \begin{enumerate}
        \item $\overline{b_x} = b_x$,
        \item $b_x = \dl_x + \sum_{y < x} h_{y, x} \dl_y,  \; \text{for}\;  h_{y, x} \in v\mathbb{Z}[v]$.
    \end{enumerate}
    The set $\{b_x \mid x \in W\}$ forms a $\mathbb{Z}[v, v^{-1}]$-basis of $H$ called the Kazhdan-Lusztig basis.
\end{thm}

\subsection{Trace and bilinear form on H}
\label{Trace and bilinear form on $H$}

We define the \textit{standard trace} $\epsilon : H \rightarrow \mathbb{Z}[v, v^{-1}]$ as the $\mathbb{Z}[v, v^{-1}]$-linear map satisfying $\epsilon(\sum c_x \dl_x) = c_{id}$. Let $i$ be the $\mathbb{Z}[v, v^{-1}]$-linear anti-involution on $H$ determined by $i(\dl_x) = \dl_{x^{-1}}$. We define a bilinear form on $H$ by
\begin{equation} \label{bilinear form}
    \lgl x, y \rgl := \epsilon(i(x) y).
\end{equation}
The standard basis is orthonormal with respect to this form, i.e., $\lgl \dl_x, \dl_y \rgl = \delta_{x,y}$ (see \cite[Lem.~3.15]{textbook}).

\begin{rem} \label{anti remark}
    While $i$ does not generally fix KL-basis elements (e.g., in type $A_2$, $i(b_{st}) = b_{ts}$), it does fix the longest basis element $b_{w_0}$. This is because by Proposition 2.9 of \cite[Prop.~2.9]{Combinatoric}, we have \begin{equation} \label{bw0}
        b_{w_0} = \sum_{x \in W} v^{l(w_0) - l(x)} \dl_x.
    \end{equation} Note that \cite{Singular} defines this form as $\epsilon(x \cdot i(y))$; however, since the standard basis is orthonormal under both definitions, they coincide.
\end{rem}

\subsection{The spherical and anti-spherical modules} \label{The spherical and anti-spherical modules}

Let $J \subset S$ be a finitary subset of generators, and $W_J$ the corresponding finite parabolic subgroup. Let $\jw$ denote the set of minimal coset representatives (m.c.r.'s) for $W_J \backslash W$.

\begin{lem}[{\cite[p.~31]{Bourbaki}}] \label{decomp} 
    Given $w \in W$, there is a unique decomposition $w = uz$, where $u \in W_J$ and $z \in \jw$. Moreover, $l(w) = l(u) + l(z)$. 
\end{lem}
\begin{lem}[{\cite[Lem.~3.1, 3.2]{Deodhar1}}] \label{wall-crossing}
    If $x \in \jw$ but $xs \notin \jw$ for some $s \in S$, then
    \begin{enumerate}[label=(\alph*)]
        \item $xs > x$
        \item $xs = tx$ for some $t \in J$.
    \end{enumerate}
\end{lem}

\noindent
Let $\cL := \Z[v, v^{-1}]$. We regard $\cL$ as a right $H_J$-module\footnote{Here, $H_J := H(W_J, J).$} $\cL(u)$ via $\delta_s \mapsto u$ for $u \in \{-v, v^{-1}\}$. We define the \textit{anti-spherical} and \textit{spherical} modules respectively as:
$$ N(J) := \cL(-v) \otimes_{H_J} H \quad \text{and} \quad M(J) := \cL(v^{-1}) \otimes_{H_J} H. $$

We focus on the spherical module $M(J)$. From now on, we will suppress `$J$' from notation and simply write $M$ in place of $M(J)$. For $x \in \jw$, set $m_x := 1 \otimes \delta_x$. The set $\{m_x \mid x \in \jw\}$ forms the \textit{standard basis} of $M$. The action of $H$ on this basis is given by:
\begin{equation} \label{smult}
m_x b_s = \begin{cases}
    m_{xs} + v^{-1} m_x & \text{if } xs < x, \; xs \in \jw\\
    m_{xs} + vm_{x} & \text{if } xs > x, \; xs \in \jw\\
    (v+v^{-1})m_{x} & \text{if } xs \notin \jw.
\end{cases} \end{equation}

Let $w_J$ be the longest element of $W_J$. The module $M$ can be realized as a right ideal of $H$.
\begin{prop} \label{smodule}
    There is an isomorphism of right $H$-modules $\phi : M \xrightarrow{\sim} b_{w_J} H$ defined by $\phi(1 \otimes h) := b_{w_J} h$.
\end{prop}

\begin{proof}
    The map is well-defined and surjective because $b_{w_J}$ satisfies the same relations as the generator $1 \otimes 1 \in M$ (see \cite[Eq.~2.8]{Singular}). For injectivity, observe that due to Lemma \ref{wall-crossing} (a), we have $\phi(m_x) = b_{w_J}\delta_x = \delta_{w_J x} + \sum_{y < w_J x} h_y \delta_y$. The set $\{\phi(m_x)\}_{x \in \jw}$ is therefore linearly independent in $H$ by upper-triangularity with respect to the standard basis.
\end{proof}

\begin{defn}
    We define the \textit{Hilbert polynomial} $\pi(J) := v^{-d_J}\sum_{x \in W_J} v^{2l(x)}$, where $d_J = l(w_J)$. It satisfies $b_{w_J}^2 = \pi(J) b_{w_J}$ (see \cite[Eq.~2.9]{Singular}).
\end{defn} 

\begin{defn}
    We define a bar involution on $M$ by $\overline{p \otimes h} := \overline{p} \otimes \overline{h}$.
\end{defn}

\begin{thm}[{\cite[Prop.~3.2]{Deodhar2}}]
    The spherical module $M$ has a \textit{Kazhdan-Lusztig basis} $\{c_x \mid x \in \jw\}$, uniquely characterized by $\overline{c_x} = c_x$ and $c_x \in m_x + \sum_{y < x} v\mathbb{Z}[v] m_y$.
\end{thm}
\subsection{The Schur algebroid}
\label{The Schur algebroid}

The \textit{Schur algebroid} $\mathcal{S}$ is a $\mathbb{Z}[v, v^{-1}]$-linear category with objects being finitary subsets $J \subset S$. The morphisms in Hom$_{\mathcal{S}}(I, J)$ are right $H$-module morphisms $M(I) \rightarrow M(J)$. By Proposition \ref{smodule}, we identify the morphisms as $\text{Hom}_{\mathcal{S}}(I, J) \cong b_{w_J} H b_{w_I}$. To ensure a well-defined composition, we associate the element $b_{w_J} h b_{w_I}$ with the morphism given by left multiplication by $b_{w_J} h b_{w_I}/\pi(I)$, where $\pi(I)$ is the Hilbert polynomial.
Composition is then given by the product $-\ast_J -: \text{Hom}_{\mathcal{S}}(J, K) \times \text{Hom}_{\mathcal{S}}(I, J) \to \text{Hom}_{\mathcal{S}}(I, K)$ defined by
\begin{equation} \label{mult}
    h_1 \ast_J h_2 := \frac{h_1 h_2}{\pi(J)}.
\end{equation}
Note that $\text{Hom}_\mathcal{S}(\varnothing, J) \cong b_{w_J}H \cong M(J)$, so that the spherical module appears as a Hom set in $\mathcal{S}$.

\subsection{Realizations} \label{realizations}

Let $\mathbbm{k}$ be an infinite commutative integral domain. We will require throughout that $\mathbbm{k}$ is a complete local ring. We fix a \textit{realization} of $(W, S)$ over $\mathbbm{k}$, consisting of a free finite-rank $\mathbbm{k}$-module $\mathfrak{h}$ together with roots $\{\al_s\}_{s \in S} \subset \mathfrak{h^*}$ and coroots $\{\al_s^{\vee}\}_{s \in S} \subset \mathfrak{h}$ satisfying
    \begin{enumerate}
        \item We have $\al_s(\al^{\vee}_s) = 2$ for all $s \in S$.
        \item For $v \in \mathfrak{h}$, the assignment $s(v) := v - \al_s(v)\al_s^{\vee}$ gives a representation of $W$.
        
    \end{enumerate}

\noindent We also have a representation of $W$ on $\mathfrak{h}^*$ via $s(\gamma) := \gamma - \gamma(\al_s^{\vee}) \al_s$ for $\gamma \in \mathfrak{h}^*$.\\

\noindent Throughout this paper, we will require that our realization satisfies some additional conditions. 
\begin{itemize}
    \item \textit{Demazure surjectivity}: The map $\al_s : \mathfrak{h} \rightarrow \mathbbm{k}$ is surjective for all $s \in S$, and evaluation at $\al_s^\vee$ gives a surjective map $\mathfrak{h}^* \rightarrow \mathbbm{k}$ for all $s \in S$.
    \item A technical condition on two-coloured quantum numbers - see condition (iii) of \cite[Defn.~5.1]{hazi2024}.
\end{itemize} Certain results will also require the following (rather strong) condition.
\begin{defn} \label{reflection faithful}
    A realization is \textit{reflection faithful} if the representation $\mathfrak{h}$ is faithful, and the set of reflections in $W$ is in bijection with the codimension 1 fixed hyperplanes of $\mathfrak{h}$. 
\end{defn}
While the standard \textit{geometric realization} $\mathfrak{h}_{\text{geom}}$ is not always reflection faithful, such realizations (e.g., $\mathfrak{h}_{KM}$) always exist \cite{Soergel, Light-leaves}. We will only require reflection-faithfulness in Section \ref{Equivalence of Diagrammatic and Algebraic Categories}, when we prove the equivalence of diagrammatic and algebraic categories. 

\subsection{Graded rings and modules}

We work in the category $A$-gBim of graded bimodules over a graded ring $A$. For a graded object $B$, we denote the \textit{grading shift} by $B(n)$, where $B(n)^i := B^{n + i}$. For a polynomial $p = \sum p_i v^i \in \mathbb{Z}[v, v^{-1}]$, we set $B^{\oplus p} := \bigoplus_{i} B(i)^{\oplus p_i}$.
The morphism spaces are defined as
$$\text{Hom}_\text{A-gBim}(M, N) := \bigoplus_{n \in \mathbb{Z}} \text{Hom}(M, N)^n,$$
where $\text{Hom}(M, N)^n$ consists of homogeneous morphisms of degree $n$.

\begin{defn}
    By a \textit{graded free} $A$-module, we mean a graded $A$-module $M$ with a basis of homogeneous elements. If $\{m_i\}$ is a basis of $M$ of homogeneous elements, then we have $$M \cong \bigoplus_{i} A m_i \cong \bigoplus_i A(-\text{deg} m_i) \cong A^{\oplus p}$$ for some polynomial $p$. We call this $p$ the \textit{graded rank} of $M$, and denote it by rk$M$. 
\end{defn}One potentially confusing aspect of this graded rank is that the degrees in the basis are the negatives of those in $p$. For example, If $M$ is a graded $A$-module generated by a single element of degree 2, then the graded rank of $M$ is $v^{-2}$. For this reason, many of our later formulae for the graded ranks of certain spaces (in Chapter \ref{Equivalence of Diagrammatic and Algebraic Categories}) have bar involutions showing up in them, since the bar involution takes $v^d$ to $v^{-d}$.  

\section{Singular Soergel Bimodules and Standard Bimodules} \label{Singular Soergel Bimodules and Standard Bimodules}

Let $R := \text{Sym}(\mathfrak{h}^*)$ be the symmetric algebra on $\mathfrak{h}^*$, viewed as a graded ring with $\deg(\mathfrak{h}^*) = 2$. For $J \subset S$, let $R^J$ denote the subring of $J$-invariants. For $s \in S$, the Demazure operator $\partial_s: R \to R^s$ is defined by $\partial_s(f) = \frac{f - s(f)}{\alpha_s}$.

\subsection{Singular Soergel bimodules}

We now define the 2-category $\mathcal{S}\mathbb{S}\text{Bim}$ of \textit{singular Soergel bimodules}, which categorifies the Schur algebroid. We first define the 2-category $\mathcal{S}\mathbb{BS}$Bim of \textit{singular Bott-Samelson bimodules}, of which $\mathcal{S}\mathbb{S}\text{Bim}$ will be the Karoubi envelope. The objects of $\mathcal{S}\mathbb{BS}$Bim are finitary subsets $J \subset S$. The 1-morphisms $J_1 \rightarrow J_2$ are the $(R^{J_2}, R^{J_1})$-bimodules (note the order) which are tensor-generated by $Ind_J^I := \tensor[_{R^I}]{R}{_{R^J}}$ and $Res_J^I := \tensor[_{R^J}]{R}{_{R^I}} (l(J) - l(I))$, where $I \subsetneq J$ are any two finitary subsets of $S$ which differ by a single simple reflection\footnote{The notation $_{R^I}R_{R^J}$ means $R$ considered as an $(R^I, R^J)$-bimodule. The notation $l(J)$ means the size of $J$.}. The 2-morphisms are bimodule morphisms. 

We can think of a 1-morphism in Hom$_{\mathcal{S}\mathbb{BS}\text{Bim}}(J, I)$ as given by a sequence $(I = I_0, I_1, ..., I_n = J)$, where $I_i$ and $I_{i + 1}$ always differ by a single element. For example, in the case where $I_1 \supset I_2 \subset I_3 \supset \cdots \supset I_{n - 1} \subset I_n$, then the corresponding bimodule is $$_{R^{I_1}} R^{I_2} \otimes_{R^{I_3}} \cdots \otimes_{R^{I_{n - 2}}} R^{I_{n - 1}}_{R^{I_n}}(d),$$ where $d$ is the number of restriction bimodules used in the product. 

We then define the 2-category $\mathcal{S}\mathbb{S}$Bim of \textit{singular Soergel bimodules}: it has the same objects as $\mathcal{S}\mathbb{BS}$Bim, and $$\text{Hom}_{\mathcal{S}\mathbb{S}\text{Bim}}(I, J) := Kar(\text{Hom}_{\mathcal{S}\mathbb{BS}\text{Bim}}(I, J)).$$

Now, a main theorem of \cite{Singular} is that the split Grothendieck group of $\mathcal{S}\mathbb{S}$Bim is isomorphic to the Schur algebroid $\mathcal{S}$. This is made precise in the theorem below (see Theorems 1.1, 1.2 and 1.4 of \cite{Singular}). We let $\isj := \text{Hom}_{\mathcal{S}}(J, I) \cong b_{w_I} H b_{w_J}$, and let $\ibj := \text{Hom}_{\mathcal{S}\mathbb{S}\text{Bim}}(J, I)$. We will write $\mathcal{B} := \tensor*[^{\varnothing}]{\mathcal{B}}{^{\varnothing}}$, and $\prescript{I}{}{\mathcal{B}} := \tensor*[^{I}]{\mathcal{B}}{^{\varnothing}}$.

\begin{thm} \label{Singular thm}
Suppose our realization $\mathfrak{h}$ is reflection faithful.
    \begin{enumerate}[label = (\alph*)]
        \item Each $\isj$ has a basis $\{\tensor*[^I]{b}{^J_p} \; | \; p \in W_I\backslash W/W_J\}$ indexed by double cosets, which we call the Kazhdan-Lusztig basis.
        \item The indecomposables of $\ibj$ are also indexed by double cosets $W_I \backslash W/W_J$. They are denoted $\tensor*[^I]{B}{^J_p}$ for $p \in W_I \backslash W/W_J$.
         \item We have a character map $ch: \ibj \rightarrow \isj$ such that $ch : [\ibj]_{\oplus} \xrightarrow{\sim} \isj$ is an isomorphism of abelian groups, and such that the following diagram commutes.
        \begin{equation} \label{singular com}
            \input{fig111.tikz}
        \end{equation}
    \end{enumerate}
\end{thm}
The endomorphism category $\mathbb{S}\text{Bim} := \text{Hom}_{\mathcal{S}\mathbb{S}\text{Bim}}(\varnothing, \varnothing)$ is the category of \textit{Soergel bimodules}, categorifying $H$. The map ch: $\mathbb{S}$Bim $\rightarrow H$ sends $B_s$ to $b_s$, for $s \in S$\footnote{Here we are using the shorthand $B_s := {}^{\varnothing} B_s^{\varnothing}$. This agrees with Definition \ref{Bs} below.}.
\begin{conj} \label{Soergel's Conjecture} (Soergel's conjecture) 
    For all $x \in W$, ch$(B_x) = b_x$.\footnote{This conjecture was proved by Elias and Williamson in \cite{Hodge} for certain reflection faithful realizations in characteristic 0. In the case that $\mathbbm{k}$ is of characteristic $p>0$, Soergel's conjecture fails, and the image of the $[B_w]$'s under the character map is known as the $p$-canonical basis.}
\end{conj}

\begin{thm}[{\cite[Thm.~1.4]{Singular}}]
    In the case where Conjecture \ref{Soergel's Conjecture} is true, then the map ch: $^I \mathcal{B}^J \rightarrow \isj$ sends $\tensor*[^I]{B}{^J_p}$ to $\tensor*[^I]{b}{^J_p}$.
\end{thm}

The Hom-category $\text{Hom}_{\mathcal{S}\mathbb{S}\text{Bim}}(\varnothing, J)$ is called the \textit{algebraic spherical category}, categorifying $M(J)$. It is the Karoubi envelope of the category of \textit{$J$-singular Bott-Samelson bimodules}, $J\mathbb{BS}$Bim\footnote{See Claim 2.4 of \cite{Thicker}.}.
\begin{defn}\label{Bs}
    For $s \in S$, let $B_s := R \otimes_{R^s} R(1)$. For an expression $\un{w} = (s_1, \dots, s_n)$, the corresponding Bott-Samelson bimodule is $BS(\un{w}) := B_{s_1} \otimes_R \cdots \otimes_R B_{s_n}$.
    We define $J\mathbb{BS}$Bim to be the full subcategory of $(R^J, R)$-bimodules whose objects are $BS(\un{w})$ restricted to $R^J$ on the left, denoted $_{R^J}BS(\un{w})$.
\end{defn}

\subsection{Standard bimodules and Localization}

For $x \in W$, the \textit{standard bimodule} $R_x$ is $R$ as a left module, with twisted right action $r \cdot_x s := r x(s)$. We define $JStd$Bim to be the category of direct sums of grading shifts of $(R^J, R)$-bimodules $R_x$.

Let $Q$ be the field of fractions of $R$. Localization induces a functor $(-)_Q := (-) \otimes_R Q: J\mathbb{BS}\text{Bim} \to (Q^J, Q)\text{-Bim}$.
For a standard bimodule, $R_x \otimes_R Q \cong Q_x$.
The localization of Bott-Samelson bimodules is well-understood:
\begin{prop} \label{localize BS}
    For an expression $\un{w}$, we have an isomorphism in $(Q^J, Q)\text{-Bim}$:
    $$ _{R^J}BS(\un{w}) \otimes_R Q \cong \bigoplus_{\un{e} \subset \un{w}} Q_{\un{w}^{\un{e}}}. $$
\end{prop}
\begin{proof}
    Combine Deodhar's formula (\cite[Lem.~3.36]{textbook}) with \cite[Lem.~6.10]{Soergel}.
\end{proof}

We denote the essential image of the localization functors by $J\mathbb{BS}\text{Bim}_Q$ and $JStd\text{Bim}_Q$. If our realization is faithful, then morphisms in the localized standard category are simple:
\begin{lem} \label{hom2}
    Suppose our realization is faithful. Then, for $x, y \in W$, we have
    \begin{equation}
    \text{Hom}_{JStd\text{Bim}_Q}(Q_x, Q_y) \cong
    \begin{cases}
        Q & \text{if } W_J x = W_J y\\
        0 & \text{otherwise}.
    \end{cases}
    \end{equation}
\end{lem}
\begin{proof}
    This follows from the faithfulness of the realization, so that $Q/Q^J$ is a Galois extension with Galois group $W_J$.
\end{proof}

\begin{rem} \label{Q index}
    Note that this means we have $Q_x \cong Q_y$ when $W_J x = W_J y$. So the objects $Q_w$ can be indexed by $w \in {}^J W$ up to isomorphism.
\end{rem}

\begin{rem} \label{non-faithful}
    In the case where the realization is not faithful, then there may be non-zero morphisms $\phi: Q_x \rightarrow Q_y$ even in the case where $W_Jx \neq W_Jy$. However, we will still have that, for any $q \in Q_x$, $\phi(q) = \phi(1 \cdot_x x^{-1}(q)) = \phi(1) \cdot_y x^{-1}(q) = \phi(1)yx^{-1}(q) $, so that $\phi$ is determined by $\phi(1)$, i.e. $\text{Hom}(Q_x, Q_y) \cong Q$.  
\end{rem}
\section{Spherical Diagrammatics} 
\label{Spherical Diagrammatics}

\subsection{The diagrammatic category \texorpdfstring{$\mathcal{M}_{BS}$}{M\_BS}}
\label{The diagrammatic category}
In \cite{SoergelCalculus}, Elias and Williamson introduced a diagrammatic category $\mathcal{D}$, which is equivalent to the category $\mathbb{BS}$Bim of Bott-Samelson bimodules\footnote{The Karoubian closure $Kar(\mathcal{D})$ is the \textit{diagrammatic Hecke category}.}. We now present a generalization of this structure: a diagrammatic category $\mathcal{M}_{BS}$ associated to a Coxeter system $(W,S)$ and a finitary subset $J \subset S$. This category provides a diagrammatic description of $J\mathbb{BS}$Bim and recovers the original category $\mathcal{D}$ when $J = \varnothing$. This category was constructed in type A by Elias in \cite{Thicker}\footnote{It is denoted $\mathcal{T}_J$ in \cite{Thicker}}, and we extend the definition here to all types without significant change.

\begin{defn}
    The category $\mathcal{M}_{BS}$ is defined as follows: Objects are expressions $\un{x}$ in $W$. To each $s \in S$ we associate a colour. Then, morphisms $\un{x} \rightarrow \un{y}$ are coloured string diagrams with bottom boundary $\un{x}$ and top boundary $\un{y}$, and with a ``wall" on the left which strings in $J$ can ``plug into". The vertices of the string diagrams (other than top and bottom boundary, and wall plug-ins) are either dots, trivalent vertices, or $2m_{st}$-valent vertices for colours $s$ and $t$:

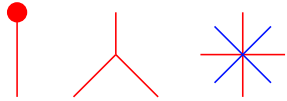
    \begin{center}
        \scalebox{1.5}{\input{fig148.tikz}}
    \end{center}
    These diagrams are considered up to isotopy. They can contain polynomials in $R$ in any of the empty regions. The diagrams are subject to the same local relations as $\mathcal{D}$ (see section 5.2 of \cite{SoergelCalculus}), together with additional `wall relations'.\\
    \\
\vspace{3mm}
\noindent
\textbf{Wall relations:}
\begin{center}
\begin{equation}
    \input{wrel1-shade.tikz}
\end{equation}
\begin{equation}
    \input{wrel2-shade.tikz}
\end{equation}
\begin{equation}
    \input{wrel3-shade.tikz}
\end{equation}
\end{center}
The final relation above is shown for $m_{st} = 3$, but applies for any $m_{st}$. This concludes the definition of $\mathcal{M}_{BS}$.
\end{defn}
The wall relations above, together with the relations inherited from $\mathcal{D}$, imply further relations:
\begin{equation} \label{wall-bubble}
    \scalebox{0.7}{\input{fig43.tikz}}
\end{equation}
\begin{equation} \label{fork-death}
    \scalebox{0.7}{\input{fig132.tikz}}
\end{equation}
\begin{equation} \label{wall-fusion}
    \scalebox{0.7}{\input{fig133.tikz}}
\end{equation}

We have a functor $\tilde{\mathcal{A}}: \mathcal{M}_{BS} \rightarrow J\mathbb{BS}\text{Bim}$. It sends $\un{x}$ to $_{R^J}BS(\un{x})$. On morphisms other than wall plug-ins, $\tilde{\mathcal{A}}$ is defined in the introduction of \cite{SoergelCalculus}. On wall plug-ins it is defined as follows.
\begin{equation} \label{new generators}
    \scalebox{0.7}{\input{fig113.tikz}}
\end{equation}

Thus we consider the wall plug-in to be a degree -1 generator. We will see in Theorem \ref{equivalence} that $\tilde{\mathcal{A}}$ is an equivalence.
\subsection{Spherical light-leaves} \label{spherical light-leaves}

We construct a basis for $\text{Hom}_{\mathcal{M}_{BS}}(\un{x}, \un{y})$ using a variation of the light-leaves algorithm \cite{SoergelCalculus}.
Given an expression $\un{x} = (s_1, ..., s_n)$ and subexpression $\un{e} \subset \un{x}$, we consider the \textit{coset stroll} $z_0, ..., z_n$, where $z_j$ is the minimal coset representative of $W_J (s_1^{e_1} \cdots s_j^{e_j})$. This stroll is decorated with a sequence $d_1', ..., d_n' \in \{U0, U1, D0, D1, X0, X1\}$ determined by the rule:
\begin{equation} \label{d_i'}
    d_i' = \begin{cases}
    Ue_i & \text{if } z_{i - 1} s_i > z_{i-1} \text{ and } z_{i - 1} s_i \in \jw\\
    De_i & \text{if } z_{i - 1} s_i < z_{i-1} \text{ and } z_{i - 1} s_i \in \jw\\
    Xe_i & \text{if } z_{i - 1} s_i \notin \jw
\end{cases}
\end{equation}
We will construct a light-leaf map $SLL_{\un{x}, \un{e}}: \un{x} \rightarrow \un{z}$ inductively: at each step $k$ we construct a map $SLL_k : \un{x}_{\leq k} \rightarrow \un{z}_k$, where $\un{x}_{\leq k} := (s_1, ..., s_k)$ and $\un{z}_k$ is a reduced expression for $z_k$. We call these expressions $\un{z}_k$ \textit{intermediate expressions}. These $SLL_k$'s will be built out of maps $\phi_k: (\un{z}_{k-1}, s_k) \rightarrow \un{z}_k$ in exactly the same way as in the non-spherical case (see \cite{SoergelCalculus}). For $d_k' \in \{U0, U1, D0, D1\}$, the map $\phi_k$ is identical to the non-spherical case \cite{SoergelCalculus}. If $d_k' = X0$, the map is identical to $U0$. The only new case is $d_k' = X1$:
\begin{equation} \label{X1 algorithm}
    \input{fig36.tikz}
\end{equation}
Since $d_k' = X1$, by Lemma \ref{wall-crossing} (b) we have $z_{k - 1} s_k = t z_{k - 1}$ for some $t \in J$ (blue strand). We apply a rex move\footnote{A \textit{rex move} is a diagram involving only $2m_{st}$-valent vertices. They correspond to sequences of braid relations. By Matsumoto's Theorem, we can get from any reduced expression to any other via rex moves.} $\beta$ to pull $t$ to the left, plug $t$ into the wall, and apply a rex move $\alpha$ to achieve the target $\un{z}_k$.

\begin{exmp}
    Let $W = S_3$, $J = \{s\}$, $\un{x} = (t, s, t)$ and $\un{e} = (1, 1, 1)$. The coset stroll is $id, t, ts, ts$ with decoration $U1, U1, X1$. The first two steps are standard. For the third step ($X1$), since $tst$ is not a minimal coset representative ($tst=sts$), we apply a rex move $\beta$ (the 6-valent vertex) and plug the resulting $s$-strand into the wall:
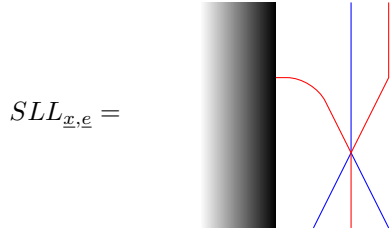
\begin{center}
    \begin{tikzpicture}
        \shade[left color = white, right color = black] (0, 0) rectangle (1, 3);
        \draw[blue] (1.5, 0) -- (2, 1) -- (2, 3);
        \draw[blue] (2.5, 0) -- (2, 1);
        \draw[red] (2, 0) -- (2, 1) -- (2.5, 2) -- (2.5, 3);
        \draw[red, rounded corners = 10] (2, 1) -- (1.5, 2) -- (1, 2);
        \node at (-1.8, 1.5) {$SLL_{\un{x}, \un{e}} = $};
    \end{tikzpicture}
\end{center}
\end{exmp}

\subsection{The double-leaves basis} \label{sdl}

We construct the basis by gluing light-leaves. For each $z \in \jw$, fix a reduced expression $\un{z}$. Given subexpressions $\un{e} \subset \un{x}$ and $\un{f} \subset \un{y}$ such that $W_J \un{x}^{\un{e}} = W_J \un{y}^{\un{f}} = W_J z$, we define the \textit{spherical double-leaf} map by composing the light-leaf for $\un{e}$ with the inverted light-leaf\footnote{I.e. $\overline{SLL}_{\un{y}, \un{f}}$ denotes the light-leaf turned upside-down.} for $\un{f}$:
\[
\mathbb{SDL}_{\un{f}, \un{e}} := \overline{SLL}_{\un{y}, \un{f}} \circ SLL_{\un{x}, \un{e}} : \un{x} \rightarrow \un{z} \rightarrow \un{y}.
\]

\begin{thm} \label{basis}
    The collection $\mathcal{SDL}_{\un{x}, \un{y}} = \{ \mathbb{SDL}_{\un{f}, \un{e}} \mid W_J \un{x}^{\un{e}} = W_J \un{y}^{\un{f}} \}$ is a basis for $\text{Hom}_{\mathcal{M}_{BS}}(\un{x}, \un{y})$ as a right $R$-module.
\end{thm}
\noindent We will prove this in Chapters \ref{Linear Independence} and \ref{Spanning}.
\begin{rem} \label{flex}
    There is some ambiguity here: there is flexibility in the construction of the light-leaves in the choice of rex moves and intermediate expressions. When constructing the set $\mathcal{SDL}_{\un{x}, \un{y}}$, we first fix a choice of $SLL_{\un{x}, \un{e}}$ for each pair $\un{x}, \un{e}$. Then the Theorem says we obtain a basis no matter how we choose the light-leaves.
\end{rem}
\section{Standard Diagrammatics} \label{Standard Diagrammatics}

In Chapter \ref{Linear Independence} we will prove the linear independence of the spherical double-leaves. We will do this by first localizing the double-leaves; that is, passing to the category $\mathcal{M}_{BS, Q} := \mathcal{M}_{BS} \otimes_R Q$. Then, we will pass to the Karoubi envelope $Kar(\mathcal{M}_{BS, Q})$, where our objects will decompose nicely and morphisms will be easy to work with.

 In this chapter, we will establish the following commutative diagram (where $\mathcal{M}_Q^{std}$ is a diagrammatic version of $JStd$Bim$_Q$, which we will define below):

\begin{equation} \label{commutative}
    \input{fig91.tikz}
\end{equation}

Moreover, we will show that the functor $\widetilde{\mathcal{S}td}$ is an equivalence of categories, and that when our realization is faithful, the other two functors in (\ref{commutative}) are also equivalences. This will allow us to infer the properties of $Kar(\mathcal{M}_{BS, Q})$ that will be necessary in proving linear independence of the double-leaves in the next chapter. 
\begin{defn}
    The category $\mathcal{M}^{std}$ is defined just as $\mathcal{D}^{std}$ (see \cite[Def.~4.3]{SoergelCalculus}), with the addition of ``wall generators" subject to the following relations.
    
    \vspace{1em} 
    
    \noindent\begin{minipage}{\linewidth} 
        \begin{minipage}[t]{0.28\linewidth}
            \vspace{0pt} 
            \begin{equation}\label{mrel1}\centering
              \scalebox{0.7}{\input{fig94.tikz}}
            \end{equation}
        \end{minipage}\hfill
        \begin{minipage}[t]{0.28\linewidth}
            \vspace{0pt}
            \begin{equation}\label{mrel2}\centering
              \scalebox{0.7}{\input{fig95.tikz}}
            \end{equation}
        \end{minipage}\hfill
        \begin{minipage}[t]{0.28\linewidth}
            \vspace{0pt}
            \begin{equation}\label{mrel3}\centering
              \scalebox{0.7}{\input{fig96.tikz}}
            \end{equation}
        \end{minipage}
    \end{minipage}

\end{defn}

Note that when $J = \varnothing$, we have that $\mathcal{M}^{std} = \mathcal{D}^{std}$. We set $\mathcal{M}_Q^{std} := \mathcal{M}^{std} \otimes Q$, allowing elements of $Q$ in the regions of our diagrams. 

\begin{defn} We define a functor $\tilde{\mathcal{F}}^{\text{std}}: \mathcal{M}^{\text{std}} \rightarrow JStd$Bim as follows. On $\mathcal{D}^{\text{std}} \subset \mathcal{M}^{\text{std}}$ it is given by the functor $\mathcal{F}^{\text{std}} : \mathcal{D}^{\text{std}} \rightarrow Std$Bim defined in \cite[Def.~4.9]{SoergelCalculus}. On the wall plug-in it is defined as follows.
\begin{equation} \label{F wall-generators}
    \scalebox{0.7}{\input{fig97.tikz}}
\end{equation}
The functor $\tilde{\mathcal{F}}_Q^{std} : \mathcal{M}_Q^{std} \rightarrow JStd\text{Bim}_Q$ is given by localization $\tilde{\mathcal{F}}^{std} \otimes Q$.
\end{defn}

\begin{defn}
    Let $\mathcal{M}_{BS, Q} := \mathcal{M}_{BS} \otimes Q$ be the same as $\mathcal{M}_{BS}$, except we are allowed to have elements of $Q$ on the right of our diagrams. This means we are allowed elements in $Q$ anywhere in our diagrams, since using the polynomial forcing relation  \cite[Eq.~5.2]{SoergelCalculus} we have:
\begin{equation} \label{loc-arg}
    \scalebox{0.7}{\input{fig64.tikz}}
\end{equation}
\end{defn}

We construct a diagrammatic category $Kar(\mathcal{M}_{BS, Q})$. We will see in Proposition \ref{Std} that $Kar(\mathcal{M}_{BS, Q}) \cong \mathcal{M}_Q^{std}$, and the latter is idempotent complete, so that $Kar(\mathcal{M}_{BS, Q})$ is indeed the Karoubi envelope of $\mathcal{M}_{BS, Q}$.

\begin{defn}
    To construct $Kar(\mathcal{M}_{BS, Q})$, take $\mathcal{M}_{BS, Q}$ and include new objects $s^{std}$ for $s \in S$, called reflection indices. The identity on $s^{std}$ is represented by a dashed line, as in $\mathcal{M}^{std}$. We then include new morphisms $\pi : s \rightarrow s^{\text{std}}$ and $\iota : s^{\text{std}} \rightarrow s$ in $Kar(\mathcal{M}_{BS, Q})$, which we call \textit{bivalent vertices}. 
    \begin{center}
        \scalebox{0.7}{\input{fig150.tikz}}
    \end{center}
    The morphisms $\pi$ and $\iota$ satisfy the following relations.
\begin{equation} \label{Krel1}
    \scalebox{0.7}{\input{fig67.tikz}}
\end{equation}
\begin{equation} \label{Krel2}
    \scalebox{0.7}{\input{fig68.tikz}}
\end{equation}
 Note that we are not considering our diagrams up to isotopy\footnote{More precisely, solid strands are allowed to move under isotopy, just as in $\mathcal{M}_{BS, Q}$, whereas dotted strands and the maps $\pi, \; \iota$ are not.}. In fact, in the end isotopic diagrams in $Kar(\mathcal{M}_{BS, Q})$ will only be the same up to sign. This concludes the definition of $Kar(\mathcal{M}_{BS, Q})$.
\end{defn}

\begin{rem} \label{KarDecomp}
    The morphisms $\pi$ and $\iota$\footnote{More precisely, we would need a factor of $\frac{1}{\al_s}$ in $\pi$ or $\iota$ to be the projection and inclusion morphisms.} express projection and inclusion of a direct summand $s^{std}$ of $s$. By putting these morphisms next to one-another, we can see that any object of $Kar(\mathcal{M}_{BS, Q})$ decomposes as a direct sum of sequences of reflection indices. In particular, for any $\un{x} \in \mathcal{M}_{BS, Q} \subset Kar(\mathcal{M}_{BS, Q})$, we have that 
\begin{equation}
    \un{x} \cong \bigoplus_{\un{e} \subset \un{x}} Q_{\un{e}},
\end{equation}
where $Q_{\un{e}}$ is the object in $Kar(\mathcal{M}_{BS, Q})$ given by the tensor product of reflection indices $s_k^{std}$ for those $k$ such that $e_k = 1$\footnote{This is abusive notation, since we are confusing an object of $Kar(\mathcal{M}_{BS, Q})$ with its image in $JStd$Bim under $\tilde{\mathcal{G}}$ (see Definition \ref{Geet}). However, other notations were more confusing.}. This decomposition will be crucial in proving linear independence of double-leaves.
\end{rem}

Using the above relations, we obtain the following crucial consequence:
\begin{equation} \label{zero-dot}
    \scalebox{0.7}{\input{fig75.tikz}}
\end{equation}
Combining (\ref{zero-dot}) and (\ref{Krel1}) with polynomial forcing, we can move polynomials across dotted strands in the same way as in $\mathcal{M}^{\text{std}}$ (see \cite[Eq.~4.1]{SoergelCalculus}).
\begin{defn} \label{Geet}
    The functor $\tilde{\mathcal{G}} : Kar(\mathcal{M}_{BS, Q}) \rightarrow JStd\text{Bim}_Q$ is defined as follows. On $\mathcal{M}_{BS, Q} \subset Kar(\mathcal{M}_{BS, Q})$, $\tilde{\mathcal{G}}$ is defined by composing $\tilde{\mathcal{A}} \otimes Q: \mathcal{M}_{BS, Q} \rightarrow J\mathbb{BS}\text{Bim}_Q$ (for $\tilde{\mathcal{A}}$ described in Section \ref{The diagrammatic category}) with $J\mathbb{BS}\text{Bim}_Q \hookrightarrow JStd\text{Bim}_Q$. Then, $\tilde{\mathcal{G}}$ sends $s^{std}$ to $Q_s$. The images of $\pi$ and $\iota$ are given below.
    \begin{equation} \label{bivalent}
    \scalebox{0.7}{\input{fig66.tikz}}
\end{equation}
Here, $d_s := \frac{1}{2}(\al_s \otimes 1 - 1\otimes \al_s)$. This concludes the definition of $\tilde{\mathcal{G}}$.
\end{defn}

The final map in (\ref{commutative}) we need to define is $\widetilde{\mathcal{S}td} : \mathcal{M}_Q^{std} \rightarrow Kar(\mathcal{M}_{BS, Q})$. On objects the functor is simple: it sends $s^{std}$ to $s^{std}$. To describe where $\widetilde{\mathcal{S}td}$ sends dashed cups and caps, we introduce some diagrammatic short-hand into $Kar(\mathcal{M}_{BS, Q})$. Include dashed cups and caps as ``new" generators in $Kar(\mathcal{M}_{BS, Q})$\footnote{We will see shortly that they can be expressed using morphisms already in $Kar(\mathcal{M}_{BS, Q})$, so that we are not changing the definition of $Kar(\mathcal{M}_{BS, Q})$, simply adding a diagrammatic short-hand.}. These are subject to the relations
\begin{equation}
    \scalebox{0.7}{\input{fig61.tikz}}
\end{equation}
\begin{equation}\label{left-slide}
    \scalebox{0.7}{\input{fig77.tikz}}
\end{equation}
\begin{equation}\label{right-slide}
    \scalebox{0.7}{\input{fig78.tikz}}
\end{equation}
From these we can see that by adding dashed caps and cups, we are not actually adding anything new to our category. For example, we can re-express a dashed cap as
\begin{center}
    \scalebox{0.7}{\input{fig76.tikz}}
\end{center}
Furthermore, these relations tell us that isotopic diagrams in $Kar(\mathcal{M}_{BS, Q})$ are only the same up to sign. For example, 
\begin{center} \label{sign}
    \scalebox{0.7}{\input{fig79.tikz}}
\end{center}
We set these dashed cups and caps to be the images of dashed cups and caps under $\widetilde{\mathcal{S}td}$. Similarly, in $Kar(\mathcal{M}_{BS, Q})$ we set 
\begin{equation} \label{2m}
    \scalebox{0.7}{\input{fig83.tikz}}
\end{equation}
where $\rho$ is the product of \textit{positive roots} corresponding to the dihedral group $\lgl s, t \rgl$ (see Section 3.4 of \cite{two-color} for a definition of positive roots). We set the above diagram to be the image of the dashed $2m_{st}$-valent vertex under $\widetilde{\mathcal{S}td}$. See section 5.4 of \cite{SoergelCalculus} for a proof that this assignment respects isotopy in $\mathcal{M}_Q^{std}$.

Finally, we define the images of the wall plug-ins in $\mathcal{M}_Q^{std}$ as follows.
\begin{equation} \label{Std wall-generators}
    \scalebox{0.7}{\input{fig101.tikz}}
\end{equation}
We have thus defined all categories and maps in diagram (\ref{commutative}). 
\begin{prop}
    The diagram (\ref{commutative}) is commutative.
\end{prop}
\begin{proof}
    This is immediate on objects. On morphisms, commutativity on $\mathcal{D}_Q^{std} \subset \mathcal{M}_{Q}^{std}$ follows from Chapter 5 of \cite{SoergelCalculus} (see the comment at the end of section 5.4). To extend this to $\mathcal{M}_Q^{std}$, we only need to show commutativity on wall generators. The images of the wall-generators under $\widetilde{\mathcal{S}td}$ were given in (\ref{Std wall-generators}). We then want to find the image of these under $\tilde{\mathcal{G}}$. Using the algebraic images given in (\ref{bivalent}) and (\ref{new generators}), we get that the image under $\tilde{\mathcal{G}}$ is 
\begin{equation} \label{G wall-generators}
    \scalebox{0.7}{\input{fig106.tikz}}
\end{equation}
which matches the images of the wall-generators under $\tilde{\mathcal{F}}_Q^{\text{std}}$ given in (\ref{F wall-generators})\footnote{The morphisms in (\ref{G wall-generators}) come from straight-forward algebraic computations using the Demazure operators.}. Therefore the diagram (\ref{commutative}) commutes.
\end{proof}

\begin{prop} \label{Ftilde}
    The functor $\tilde{\mathcal{F}}_Q^{std} : \mathcal{M}_Q^{std} \rightarrow JStd\text{Bim}_Q$ is faithful. When our realization is faithful,  $\tilde{\mathcal{F}}_Q^{std}$ is an equivalence of categories. 
\end{prop}

It is clear that $\tilde{\mathcal{F}}_Q^{\text{std}}$ is essentially surjective. To show fully-faithfulness, we use the following result for non-spherical categories.  

\begin{prop} \label{F_0 Std}
    The space $\text{Hom}_{\mathcal{D}_Q^{std}}(\un{x}, \un{y})$ is isomorphic to $Q$ when $x = y$, and is zero otherwise \cite[Thm 4.8]{SoergelCalculus}.
\end{prop}

\noindent \textit{Proof of Proposition \ref{Ftilde}.}  For any two $x, y \in W$ we have that there is at most one morphism in $\operatorname{Hom}_{JStd\text{Bim}_Q}(Q_x, Q_y)$ up to $Q$\footnote{This is because a morphism $\varphi: Q_x \rightarrow Q_y$ is determined by $\varphi(1)$.}. Let us take two expressions $\un{x}, \un{y}$ of $x, y \in W$ such that $W_J x = W_J y$. We want to show that there is only one morphism $\un{x} \rightarrow \un{y}$ in $\mathcal{M}_{ Q}^{\text{std}}$ up to $Q$, and that this does not get sent to zero under $\tilde{\mathcal{F}}_Q^{std}$. We will do this by constructing a specific morphism, and then showing that any other morphism $\un{x} \rightarrow \un{y}$ is the same as the one we constructed. 

We construct a morphism $\phi : \un{x} \rightarrow \un{y}$ as follows. Let $x = uz$ and $y = u'z$, for $u, u' \in W_J$ and $z \in \jw$. We know these decompositions are unique by Lemma \ref{decomp}. Set $v := u'u^{-1}$, and fix a rex $\un{v}$ of $v$. Then $v x = v u z = u' z = y$. This means that $(\un{v}, \un{x})$ and $\un{y}$ are both expressions for $y$. Therefore by Proposition \ref{F_0 Std}, there is a unique morphism (up to $Q$) $\psi$ in $\mathcal{D}_{Q}^{\text{std}}$ from $(\un{v}, \un{x})$ to $\un{y}$. We construct our morphism $\phi: \un{x} \rightarrow \un{y}$ as in the following diagram.
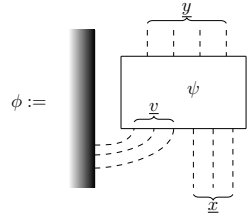
\begin{center}
    \scalebox{0.7}{\input{fig98.tikz}}
\end{center}
Now, consider an arbitrary non-zero morphism $\un{x} \rightarrow \un{y}$ in $\mathcal{M}_{Q}^{\text{std}}$. We can assume without loss of generality that there are no $Q$-boxes, since these can all be shifted to the right. We will shift any wall-strands to the bottom via isotopy.
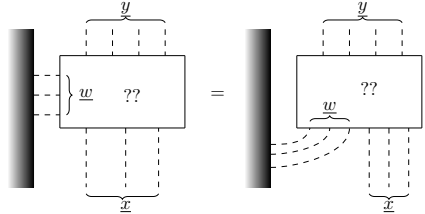
\begin{center}
    \scalebox{0.7}{\input{fig99.tikz}}
\end{center}
So we have some expression $\un{w}$ of $w \in W_J$ coming in from the wall. This means that $wx = y$, since in the box labelled `??' we have a non-zero morphism $(\un{w}, \un{x}) \rightarrow \un{y}$ in $\mathcal{D}_Q^{std}$. This means in particular that $wx = y = vx$, where $v$ is defined above, so that $w = v$. Now, we can in fact choose the box `??' to be any morphism we want in $\mathcal{D}_Q^{std}$, since they will all be the same up to $Q$. Let us choose the `??' box to be a map $\gamma: \un{w} \rightarrow \un{v}$ consisting of rex moves and caps\footnote{This is possible since $\un{v}$ is reduced.}, followed by the morphism $\psi : (\un{v}, \un{x}) \rightarrow \un{y}$ we chose above. 
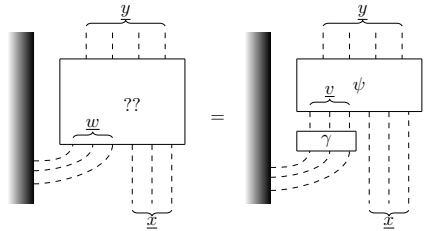
\begin{center}
    \scalebox{0.7}{\input{fig100.tikz}}
\end{center}
Since $\gamma$ consists of $2m$-valent vertices and caps, it can be pulled into the wall using relations (\ref{mrel2}) and (\ref{mrel3}). We end up with the map $\phi$. 

Thus we have seen that if $W_Jx = W_Jy$, then there is only one non-zero morphism $\un{x} \rightarrow \un{y}$ in $\mathcal{M}_{Q}^{\text{std}}$ up to scalar. This means that the morphism spaces Hom$_{\mathcal{M}_{Q}^{\text{std}}}(\un{x}, \un{y})$ and Hom$_{JStd\text{Bim}_{Q}}(Q_x, Q_y)$ are both isomorphic to $Q$. Furthermore, the image of any non-zero morphism in Hom$_{\mathcal{M}_{Q}^{\text{std}}}(\un{x}, \un{y})$ under $\tilde{\mathcal{F}}_Q^{std}$ sends the 1-tensor to the 1-tensor (times some $q \in Q$), so this is a non-zero map. 

Next, we show that if $W_J x \neq W_J y$, then there are no non-zero morphisms $\un{x} \rightarrow \un{y}$. Suppose there is a non-zero morphism $\varphi: \un{x} \rightarrow \un{y}$ in $\mathcal{M}_Q^{std}$. Then using the above diagrams, this means there is a non-zero morphism $(\un{w}, \un{x}) \rightarrow \un{y}$ in $\mathcal{D}_Q^{std}$ for some expression $\un{w}$ of $w \in W_J$, so that $wx = y$ and $W_J x = W_J y$, a contradiction. So $\operatorname{Hom}_{\mathcal{M}_Q^{std}}(\un{x}, \un{y}) = 0$ when $W_J x \neq W_J y$, so that this will map faithfully into $\operatorname{Hom}_{JStd\text{Bim}_Q}(Q_x, Q_y)$. 

Finally, in the case that our realization is faithful, Lemma \ref{hom2} tells us that $\tilde{\mathcal{F}}_Q^{std}$ is an equivalence. $\hfill \square$\\
\\
\noindent In the course of the above proof, we showed the following.

\begin{cor} \label{HomStdQ}
    The space $\operatorname{Hom}_{\mathcal{M}_Q^{std}}(\un{x}, \un{y})$ is isomorphic to $Q$ when $W_Jx = W_Jy$, and is zero otherwise.
\end{cor}

\begin{prop} \label{Std}
    The functor $\widetilde{\mathcal{S}td} : \mathcal{M}_{Q}^{\text{std}} \rightarrow Kar(\mathcal{M}_{BS, Q})$ is an equivalence of categories. 
\end{prop}

\noindent \textit{Proof.} Essential surjectivity is immediate as every object in $Kar(\cM_{BS, Q})$ decomposes into reflection indices. To show fullness, we can take any morphism in $Kar(\cM_{BS, Q})$ and move the wall-strands to the top\footnote{Recall that not all diagrams in $Kar(\cM_{BS, Q})$ are well-defined up to isotopy. For example, see (\ref{sign}). However, the problem only comes from rotating bivalent vertices, which we need not do to move these wall-strands to the top.}.
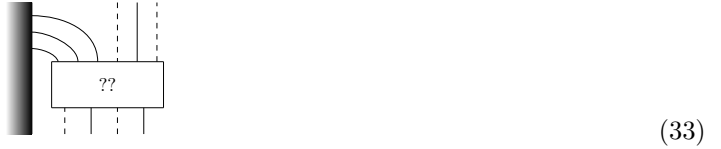
\begin{equation} \label{full}
    \scalebox{0.7}{\input{fig107.tikz}}
\end{equation}
Then, the diagram in the box labelled `??' is in the image of $\widetilde{\mathcal{S}td}$ by Proposition 5.23 of \cite{SoergelCalculus}\footnote{This gives an equivalence $\mathcal{S}td: \mathcal{D}_Q^{\text{std}} \rightarrow Kar(\mathcal{D}_Q)$.}.

Now, the solid strands plugging into the wall in the above diagram can also be seen to be in the image of $\widetilde{\mathcal{S}td}$, using the relation
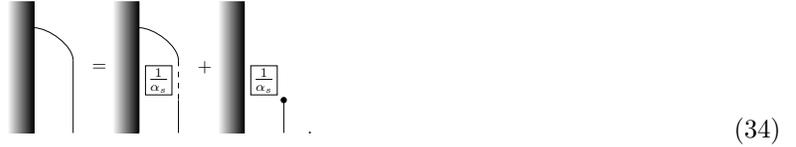
\begin{equation}
    \scalebox{0.7}{\input{fig108.tikz}}
\end{equation}
We obtained the above relation by dividing (\ref{Krel2}) by $\al_s$, and then pulling the dot into the wall. The right-hand side of this relation is in the image of $\widetilde{\mathcal{S}td}$, since the map
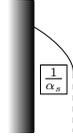
\begin{center}
    \scalebox{0.7}{\input{fig109.tikz}}
\end{center}
is the image of 
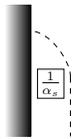
\begin{center}
    \scalebox{0.7}{\input{fig110.tikz}}
\end{center}

Finally, to show $\widetilde{\mathcal{S}td}$ is faithful, take two morphisms $\phi, \psi$ in $\mathcal{M}_{Q}^{\text{std}}$ such that $\widetilde{\mathcal{S}td}(\phi) = \widetilde{\mathcal{S}td}(\psi)$. This implies $\tilde{\mathcal{G}} \circ \widetilde{\mathcal{S}td}(\phi) = \tilde{\mathcal{G}} \circ \widetilde{\mathcal{S}td}(\psi)$, so by (\ref{commutative}) this means $\tilde{\mathcal{F}}_Q^{\text{std}}(\phi) = \tilde{\mathcal{F}}_Q^{\text{std}}(\psi)$. By Proposition \ref{Ftilde} we have $\phi = \psi$, and the functor $\widetilde{\mathcal{S}td}$ is faithful. $\hfill \square$
\begin{cor} \label{Gequivalence}
    The functor $\tilde{\mathcal{G}}$ is faithful. When our realization is faithful, $\tilde{\mathcal{G}} : Kar(\mathcal{M}_{BS, Q}) \rightarrow JStd$Bim$_Q$ is an equivalence.
\end{cor}

\section{Linear independence}\label{Linear Independence}

In this chapter, we will prove that the localized double-leaves form a basis. 

\begin{prop} \label{localized-basis}
    Given $\un{x}, \un{y} \in \mathcal{M}_{BS}$, let $\mathcal{SDL}_{\un{x}, \un{y}}$ be a set\footnote{We say \textit{a} set rather than \text{the} set, because we have flexibility in our construction of the light-leaves. See Remark \ref{flex}.} containing one double-leaf map $\mathbb{SDL}_{\un{f}, \un{e}}$ for every pair of subexpressions $\un{e} \subset \un{x}$, $\un{f} \subset \un{y}$ such that $W_J \un{x}^{\un{e}} = W_J \un{y}^{\un{f}}$. Then, after localization, this set of double-leaf maps forms a basis of Hom$_{\mathcal{M}_{BS, Q}}(\un{x}, \un{y})$ as a right $R$-module.
\end{prop}

\noindent Our double-leaf maps lie in $\mathcal{M}_{BS}$. We take the images of them under the functors
\begin{equation} \label{categories}
    \mathcal{M}_{BS} \xrightarrow{(-) \otimes Q} \mathcal{M}_{BS, Q} \hookrightarrow Kar(\mathcal{M}_{BS, Q}).
\end{equation} 
We will show that these images form a basis of the Hom spaces in $Kar(\mathcal{M}_{BS, Q})$. Then, because Hom$_{\mathcal{M}_{BS, Q}}(\un{x}, \un{y}) \cong $ Hom$_{Kar(\mathcal{M}_{BS, Q})}(\un{x}, \un{y})$, this implies that the localized double-leaves form a basis of $\mathcal{M}_{BS, Q}$. This in turn means that the double-leaves are linearly independent in $\mathcal{M}_{BS}$. For if $$\sum  \mathbb{SDL}_{\un{f}, \un{e}} \cdot c_{\un{f}, \un{e}} = 0, \; c_{\un{f}, \un{e}} \in R,$$ then $$\sum  \mathbb{SDL}_{\un{f}, \un{e}} \cdot c_{\un{f}, \un{e}} \otimes id_Q = \sum  \mathbb{SDL}_{\un{f}, \un{e}} \otimes id_Q \cdot c_{\un{f}, \un{e}} = 0,$$ so that $c_{\un{f}, \un{e}} = 0$ for all $(\un{f}, \un{e})$.

\subsection{A partial order on pairs of subexpressions} \label{Section: partial order}

We will abuse notation by letting $\mathbb{SDL}_{\un{f}, \un{e}}$ denote the image of the double-leaf in $Kar(\mathcal{M}_{BS, Q})$ under the functors in (\ref{categories}). In $Kar(\mathcal{M}_{BS, Q})$, the double-leaf map $\mathbb{SDL}_{\un{f}, \un{e}} : \un{x} \rightarrow \un{y}$ is a map $\bigoplus_{\un{e'} \subset \un{x}} Q_{\un{e'}} \rightarrow \bigoplus_{\un{f'} \subset \un{y}} Q_{\un{f'}}$ (see Remark \ref{KarDecomp}). Then, by Proposition \ref{Std} and Corollary \ref{HomStdQ}, we can therefore think of this double-leaf map as a matrix, with entries $p_{e', f'}^{e, f} \in Q$. To be precise, we let $p_{e', f'}^{e, f}$ be the composition $Q_{\un{e'}} \hookrightarrow \un{x} \xrightarrow{\mathbb{SDL}_{\un{f}, \un{e}}} \un{y} \twoheadrightarrow Q_{\un{f'}}$. We will prove Proposition \ref{localized-basis} by showing that the maps $\mathbb{SDL}_{\un{f}, \un{e}}$ in $Kar(\mathcal{M}_{BS, Q})$ are upper-triangular with respect to a certain partial order on the set of pairs $(\un{e}, \un{f})$. To define this partial order, we first define a partial order each of the components; i.e. a partial order on subexpressions. 

Take an expression $\un{w} = (s_1, ..., s_n)$, and let $E$ be the set of subexpressions $\un{e} \subset \un{w}$. There is a partial order on $E$ called the \textit{path-dominance order}: for sub-expressions $\un{e} = (e_1, ..., e_n)$ and $\un{f} = (f_1, ..., f_n)$, we can take their strolls $x_0, ..., x_n$ and $y_0, ..., y_n$ respectively (i.e. $x_i := s_1^{e_1} \cdots s_i^{e_i}$). Then we say that $\un{f} \leq \un{e}$ in the path dominance order if $y_i \leq x_i$ in the Bruhat order, for all $i$.  

 We will put a partial order $\preceq$ on $E$ which is a slight variation of this. Rather than taking the strolls associated to $\un{e}$ and $\un{f}$, we will take coset strolls, introduced in Section \ref{spherical light-leaves}: The coset stroll $cs(\un{e})$ of $\un{e}$ is $\tilde{x}_0, ..., \tilde{x}_n$, where $\tilde{x}_i$ is the minimal coset representative of $x_i = s_1^{e_1} \cdots s_k^{e_k}$. We write $cs(\un{f}) \leq cs(\un{e})$ if $\tilde{y}_i \leq \tilde{x}_i$ for all $i$, where $\tilde{y}_0, ..., \tilde{y}_n$ and $\tilde{x}_0, ..., \tilde{x}_n$ are the coset strolls of $\un{f}$ and $\un{e}$ respectively.

 \begin{defn} \label{order}
     We define a partial order $\preceq$ on the set $E$ of subexpressions of $\un{w}$ as follows: for $\un{f}, \un{e} \subset \un{w}$, $\un{f} \preceq \un{e}$ if either $cs(\un{f}) < cs(\un{e})$, or $cs(\un{f}) = cs(\un{e})$ and there is no step $k$ at which $\un{e}$ is $X0$ and $\un{f}$ is $X1$ (see (\ref{d_i'})). 
 \end{defn}
\begin{exmp}
    Let $(W, S)$ be type $A_2$, and let $J = \{s\}$. Take $\un{w} = (s, t, s, t)$. Let $\un{e}, \un{f}, \un{g} \subset \un{w}$ be $$\un{e} = (0, 1, 1, 1), \; \un{f} = (0, 1, 1, 0), \; \un{g} = (1, 1, 1, 0).$$ The coset strolls of each of these expressions is $id, id, t, ts, ts$, so $cs(\un{e}) = cs(\un{f}) = cs(\un{g})$. The decoration of $\un{e}$ is $X0, U1, U1, X1$, while that of $\un{f}$ is $X0, U1, U1, X0$. There is no step at which $\un{e}$ is $X0$ and $\un{f}$ is $X1$, so we have $\un{f} \prec \un{e}$. Similarly, we have $\un{f} \prec \un{g}$. However, $\un{e}$ and $\un{g}$ are not comparable. 
\end{exmp}
\noindent Now we use this partial order on subexpressions to define a partial order (abusively denoted $\prec$) on \textit{pairs} of subexpressions $(\un{e}, \un{f})$, where $\un{e} \subset \un{x}$, $\un{f} \subset \un{y}$. 
\begin{defn}
    For $\un{e} , \un{e}' \subset \un{x}, \; \un{f}, \un{f}' \subset \un{y}$, we write $(\un{e}', \un{f}') \preceq (\un{e}, \un{f})$ if $\un{e}' \preceq \un{e}$ and $\un{f}' \preceq \un{f}$.
\end{defn}
\subsection{Proof of linear independence} \label{Section: Proof of linear independence}
\noindent We will now show that the maps $\mathbb{SDL}_{\un{f}, \un{e}}$ are upper triangular with respect to this partial order.
\begin{prop} \label{upper-triangular}
\begin{enumerate}[label=(\alph*)]
\item If $p_{e', f'}^{e, f} \neq 0 $, then $(\un{e'}, \un{f'}) \preceq (\un{e}, \un{f})$. 
\item Moreover, $p_{e, f}^{e, f} \neq 0$ for any $\un{e} \subset \un{x}, \; \un{f} \subset \un{y}$.
\end{enumerate}
\end{prop}
This proposition will be proven over the next few lemmas. Recall that $p_{e', f'}^{e, f}$ is given by the composition
\begin{equation} 
     Q_{\un{e'}} \hookrightarrow \un{x} \xrightarrow{SLL_{\un{x}, \un{e}}} \un{z} \xrightarrow{\overline{SLL_{\un{y}, \un{f}}}} \un{y} \twoheadrightarrow Q_{\un{f'}}
\end{equation}
where $\un{z}$ is a rex of the minimal coset representative $z$ of $W_J \un{x}^{\un{e}} = W_J \un{y}^{\un{f}}$. Consider the first half of this map (we will call this $\phi_{\un{e}, \un{e'}}$):
\begin{equation} 
     \phi_{\un{e}, \un{e}'} := Q_{\un{e'}} \hookrightarrow \un{x} \xrightarrow{SLL_{\un{x}, \un{e}}} \un{z}.
\end{equation}
We will first show the following statement:
\begin{lem} \label{first half}
If $\phi_{\un{e}, \un{e}'} \neq 0 $, then $\un{e}' \preceq \un{e}$.    
\end{lem} 
\noindent \textit{Proof.} First, consider the case where $cs(\un{e}) = cs(\un{e}')$. Suppose $\un{e}' \npreceq \un{e}$. Then there is some step $k$ at which $\un{e}'$ is X1 and $\un{e}$ is X0. At this step in the construction of $SLL_{\un{x}, \un{e}}$, the map $SLL_k$ will have an up-dot \begin{tikzpicture}
     \draw[red] (0, 0) -- (0, 0.4);
     \filldraw[red] (0, 0.4) circle(1pt);
 \end{tikzpicture}
 on the right. When we then precompose with the inclusion map $Q_{\un{f}} \hookrightarrow \un{x}$, we will end up with a strand \begin{tikzpicture}
     \draw[red, densely dashed] (0, 0) -- (0, 0.3);
     \draw[red] (0, 0.3) -- (0, 0.6);
     \filldraw[red] (0, 0.6) circle(1.5pt);
 \end{tikzpicture}
 , rendering the whole map zero. Therefore $\phi_{\un{e}, \un{e}'} = 0$.

 Now, take the case where $cs(\un{e}) \neq cs(\un{e}')$. We suppose that $\phi_{\un{e}, \un{e}'} \neq 0$, and we will show that $cs(\un{e}') < cs(\un{e})$ (and that therefore $\un{e}' \prec \un{e}$). 
 Recall the recursive construction of the spherical light-leaves:
 \begin{center}
\begin{equation} \label{eq:con}
    \scalebox{0.7}{\begin{tikzpicture}
        \draw (0, 0) -- (0, 1);
        \draw (2, 0) -- (2, 1);
        \draw (-0.3, 1) -- (2.3, 1);
        \draw (2.3, 1) -- (2.3, 3.7);
        \draw (2.3, 3.7) -- (-0.3, 3.7);
        \draw (-0.3, 3.7) -- (-0.3, 1);
        \draw (0.2, 3.7) -- (0.2, 4.7);
        \draw (1.8, 3.7) -- (1.8, 4.7);
        \node at (1, 0.5) {$\cdots$};
        \node at (1, 4.2) {$\cdots$};
        \node at (1, 2.4) {$SLL_{k}$};
        \node at (2.8, 2.4) {=};
        \draw (3.6, 0) -- (3.6, 0.7);
        \draw (3.3, 0.7) -- (5.5, 0.7);
        \draw (5.5, 0.7) -- (5.5, 2);
        \draw (5.5, 2) -- (3.3, 2);
        \draw (3.3, 2) -- (3.3, 0.7);
        \draw (5.2, 0) -- (5.2, 0.7);
        \node at (4.4, 0.4) {$\cdots$};
        \draw (3.8, 2.7) -- (6.2, 2.7) -- (6.2, 4) -- (3.8, 4) -- cycle;
        \draw (4, 2) -- (4, 2.7);
        \draw (5.3, 2) -- (5.3, 2.7);
        \draw (6, 0) -- (6, 2.7);
        \node at (4.65, 2.35) {$\cdots$};
        \draw (4.1, 4) -- (4.1, 4.7);
        \draw (5.9, 4) -- (5.9, 4.7);
        \node at (5, 4.35) {$\cdots$};
        \node at (4.4, 1.35) {$SLL_{k - 1}$};
        \node at (5, 3.35) {$\phi_k$};
        \node at (6, -0.2) {$s_k$};
        \node at (3.6, -0.2) {$s_1$};
        \node at (5.2, -0.2) {$s_{k-1}$};

    \end{tikzpicture}}
\end{equation}
\end{center}
At the end of this construction we precompose with the inclusion $\iota_{\un{e}'} : Q_{\un{e}'} \hookrightarrow \un{x}$ to get $\phi_{\un{e}, \un{e}'} = SLL_{\un{x}, \un{e}} \circ \iota_{\un{e}'}$. Now, at any step $k$, let $\un{x}_{\leq k} = (s_1, s_2, ..., s_k)$, and let $\un{e}'_{\leq k}$ be the subexpression $(e_1', e_2', e_3', ..., e_k') \subset \un{x}_{\leq k}$. Let $\iota_{\leq k} : Q_{\un{e}_{\leq k}'} \hookrightarrow \un{x}_{\leq k}$. We can see that the composition $SLL_k \circ \iota_{\leq k}$ will appear inside $SLL_{\un{x}, \un{e}} \circ \iota_{\un{e}'}$, depicted below inside the green box (for $k = n - 2$, $\un{e}' = (1, 0, 1, 0, 1, 0)$):
\begin{center}
\begin{equation} \label{eq:con2}
    \scalebox{0.7}{\begin{tikzpicture}
        \draw (-1, 0) -- (-1, 1);
        \draw (-0.4, 0) -- (-0.4, 1);
        \draw (0.2, 0) -- (0.2, 1);
        \draw (0.8, 0) -- (0.8, 1);
        \draw (1.4, 0) -- (1.4, 1);
        \draw (2, 0) -- (2, 1);
        \draw[densely dashed] (-1, 0) -- (-1, -0.8);
        \draw[densely dashed] (0.2, 0) -- (0.2, -0.8);
        \draw[densely dashed] (1.4, 0) -- (1.4, -0.8);
        \filldraw (-0.4, 0) circle(2pt);
        \filldraw (0.8, 0) circle(2pt);
        \filldraw (2, 0) circle(2pt);
        \draw (-1.3, 1) -- (2.3, 1);
        \draw (2.3, 1) -- (2.3, 4.7);
        \draw (2.3, 4.7) -- (-1.3, 4.7);
        \draw (-1.3, 4.7) -- (-1.3, 1);
        \draw (-0.8, 4.7) -- (-0.8, 6.4);
        \draw (1.8, 4.7) -- (1.8, 6.4);
        \node at (0.5, 5.5) {$\cdots$};
        \node at (0.5, 2.9) {$SLL_{\un{x}, \un{e}}$};
        \node at (2.8, 2.4) {=};
        \draw (3.6, 0) -- (3.6, 0.7);
        \draw (3.3, 0.7) -- (5.5, 0.7);
        \draw (5.5, 0.7) -- (5.5, 2);
        \draw (5.5, 2) -- (3.3, 2);
        \draw (3.3, 2) -- (3.3, 0.7);
        \draw (4.1, 0) -- (4.1, 0.7);
        \draw (4.6, 0) -- (4.6, 0.7);
        \draw (5.2, 0) -- (5.2, 0.7);
        \draw (6, 0) -- (6, 2.7);
        \draw (6.8, 0) -- (6.8, 4.7);
        \draw[densely dashed] (3.6, 0) -- (3.6, -0.8);
        \filldraw (4.1, 0) circle(2pt);
        \draw[densely dashed] (4.6, 0) -- (4.6, -0.8);
        \filldraw (5.2, 0) circle(2pt);
        \draw[densely dashed] (6, 0) -- (6, -0.8);
        \filldraw (6.8, 0) circle(2pt);
        \draw (3.8, 2.7) -- (6.2, 2.7) -- (6.2, 4) -- (3.8, 4) -- cycle;
        \draw (4, 2) -- (4, 2.7);
        \draw (5.3, 2) -- (5.3, 2.7);

        \draw[green, densely dashed] (3, -1) -- (3, 2.3) -- (5.8, 2.3) -- (5.8, -1) -- cycle;
        
        \node at (4.65, 2.35) {$\cdots$};
        \draw (4.3, 4) -- (4.3, 4.7);
        \draw (5.9, 4) -- (5.9, 4.7);
        \node at (5, 4.35) {$\cdots$};
        \node at (4.4, 1.35) {$SLL_{n - 2}$};
        \node at (5, 3.35) {$\phi_{n - 1}$};
        
        \draw (4.1, 4.7) -- (7.1, 4.7) -- (7.1, 5.8) -- (4.1, 5.8) -- cycle;
        \node at (5.5, 5.25) {$\phi_n$};
        \draw (4.5, 5.8) -- (4.5, 6.4);
        \draw (6.7, 5.8) -- (6.7, 6.4);
        \node at (5.6, 6.1) {$\cdots$};
    \end{tikzpicture}}
\end{equation}
\end{center}
From this, we can see that if any $SLL_k \circ \iota_{\leq k}$ is zero, then so is $\phi_{\un{e}, \un{e}'} = SLL_{\un{x}, \un{e}} \circ \iota_{\un{e}'}$.

Now, the map $SLL_k \circ \iota_{\leq k}$ is a map $Q_{\un{e}_{\leq k}'} \hookrightarrow \un{x}_{\leq k} \rightarrow \un{z}_k$, where $\un{z}_k$ is a rex for the minimal coset representative $z_k$ of $\un{x}_{\leq k}^{\un{e}_{\leq k}}$ (i.e. $z_0, ..., z_n$ is the coset stroll $cs(\un{e})$). For this map to be non-zero, we need that $Q_{\un{e}_{\leq k}'}$ can map to some summand $Q_{\un{g}}$ in $\un{z}_k$ (i.e. $\un{g} \subset \un{z}_k$). By Corollary \ref{HomStdQ}, this means that $W_Jy_k = W_Jw$, where $y_k := \un{x}_{\leq k}^{\un{e}_{\leq k}'}$ (i.e. $y_0, ..., y_n$ is the coset stroll of $\un{e}'$) and $w := \un{z}_{k}^{\un{g}}$. If $\tilde{y}_k$ is the minimal coset representative of $y_k$, then $W_J \tilde{y}_k = W_J y_k = W_J w$, so that $\tilde{y}_k \leq w$ (since $\tilde{y_k}$ is minimal). We have that $w \leq z_k$ (since $\un{z}_k$ is reduced), so that $\tilde{y}_k \leq w \leq z_k$. So $\tilde{y}_k \leq z_k$ for all $k$, i.e. $cs(\un{e}') \leq cs(\un{e})$. Since we assumed $cs(\un{e}') \neq cs(\un{e})$, we get $cs(\un{e}') < cs(\un{e})$, and therefore $\un{e}' \prec \un{e}$.
$\hfill \square$\\
\\
We can use this to prove part $(a)$ of Proposition \ref{upper-triangular}:\\
\\
\noindent \textit{Proof of Proposition \ref{upper-triangular} $(a)$.} Suppose $p_{e', f'}^{e, f} \neq 0$. We want to show that $(\un{e}', \un{f}') \preceq (\un{e}, \un{f})$, i.e. that $\un{e}' \preceq \un{e}$ and $\un{f}' \preceq \un{f}$. Recall (again) that $p_{e', f'}^{e, f}$ is given by the composition
\begin{equation} 
     Q_{\un{e'}} \hookrightarrow \un{x} \xrightarrow{SLL_{\un{x}, \un{e}}} \un{z} \xrightarrow{\overline{SLL_{\un{y}, \un{f}}}} \un{y} \twoheadrightarrow Q_{\un{f'}}.
\end{equation} The first half of this we called $\phi_{\un{e}, \un{e}'}$, and so if $p_{e', f'}^{e, f} \neq 0$, then $\phi_{\un{e}, \un{e}'} \neq 0$. By Lemma \ref{first half}, this means $\un{e}' \preceq \un{e}$.

Now consider the second half of the map $p_{e', f'}^{e, f}$:
$$\un{z} \xrightarrow{\overline{SLL_{\un{y}, \un{f}}}} \un{y} \twoheadrightarrow Q_{\un{f'}}.$$ Note that (up to a constant) this is just $\overline{\phi_{\un{f}, \un{f}'}}$, i.e. $\phi_{\un{f}, \un{f}'}$ flipped upside down. This is because the projection map $\un{y} \twoheadrightarrow Q_{\un{f}'}$ is (up to a constant) the inclusion map $Q_{\un{f}'} \hookrightarrow \un{y}$ flipped upside-down. So $p_{e', f'}^{e, f} \neq 0$ implies $\overline{\phi_{\un{f}, \un{f}'}} \neq 0$, which implies $\phi_{\un{f}, \un{f}'} \neq 0$. Again by Lemma \ref{first half}, this means $\un{f}' \preceq \un{f}$. Thus we have $\un{e}' \preceq \un{e}$ and $\un{f}' \preceq \un{f}$, i.e. $(\un{e}', \un{f}') \preceq (\un{e}, \un{f})$. $\hfill \square$\\
\\
To prove part $(b)$ of Proposition \ref{upper-triangular}, we need another lemma. Take any expression $\un{w}$ and a subexpression $\un{e} \subset \un{w}$. This gives us a light-leaf map $SLL_{\un{w}, \un{e}} : \un{w} \rightarrow \un{z}$, where $\un{z}$ is a reduced expression for the minimal coset rep $z$ of $\un{w}^{\un{e}}$. Then in $Kar(\mathcal{M}_{BS, Q})$, $\un{z}  \cong \bigoplus_{\un{e'} \subset \un{z}} Q_{\un{e'}}$. Let $Q_z := Q_{\un{h}}$ for $\un{h} \subset \un{z}$ the subexpression consisting of all 1's\footnote{Once again, this is abusive notation, where we are conflating an object on $Kar(\mathcal{M}_{BS, Q})$ with its image in $JStd\text{Bim}_Q$ under $\tilde{\mathcal{G}}$.}. Thus we have a projection map $p : \un{z} \rightarrow Q_z$.

Consider the composition $$ Q_{\un{e}} \hookrightarrow \un{w} \xrightarrow{SLL_{\un{w}, \un{e}}} \un{z} \twoheadrightarrow Q_z$$ Since $z$ is the minimal coset rep of $\un{w}^{\un{e}}$, we have that $W_J \un{w}^{\un{e}} = W_J z$, so that by Corollary \ref{HomStdQ} the composition above is given by some $q_{\un{e}} \in Q$.
\begin{lem} \label{q}
      For any $\un{e} \subset \un{w}$, $q_{\un{e}} \neq 0$.
\end{lem}
\noindent \textit{Proof.} This proof follows almost exactly the proof of Proposition 6.6 in \cite{SoergelCalculus}. The only difference is there may be dotted strands plugging into the wall, which are non-zero by (\ref{F wall-generators}) and (\ref{Std wall-generators}).
$\hfill \square$\\
\\
We can now prove Proposition \ref{upper-triangular} $(b)$:\\
\\
\noindent \textit{Proof of Proposition \ref{upper-triangular} $(b)$:} We want to show that $p_{e, f}^{e, f} \neq 0$. We will show that in fact, $p_{e, f}^{e, f} = c \cdot q_{\un{e}} \cdot q_{\un{f}}$, where $c \in Q$ is non-zero. Since by Lemma \ref{q} we know that $q_{\un{e}}$ and $q_{\un{f}}$ are non-zero,  this will imply $p_{e, f}^{e, f} \neq 0$.

We know that $p_{e, f}^{e, f}$ is given by the map \begin{equation} \label{p}
    Q_{\un{e}} \hookrightarrow \un{x} \xrightarrow{SLL_{\un{x}, \un{e}}} \un{z} \xrightarrow{\overline{SLL}_{\un{y}, \un{f}}} \un{y} \twoheadrightarrow Q_{\un{f}}.
\end{equation}
Let us add a projection and inclusion onto/from $Q_z$ in the middle:
\begin{equation} \label{project include}
Q_{\un{e}} \hookrightarrow \un{x} \xrightarrow{SLL_{\un{x}, \un{e}}} \un{z} \twoheadrightarrow Q_z \hookrightarrow \un{z} \xrightarrow{\overline{SLL}_{\un{y}, \un{f}}} \un{y} \twoheadrightarrow Q_{\un{f}}.
\end{equation} We claim that this does not change anything, i.e. that the maps (\ref{p}) and (\ref{project include}) are the same. To see this, look at the first two maps $Q_{\un{e}} \hookrightarrow \un{x} \rightarrow \un{z}$ in (\ref{project include}): $Q_{\un{e}}$ can only land in direct summands $Q_w$\footnote{To be precise, $Q_w := Q_{\un{g}}$ for some $\un{g} \subset \un{z}$ such that $\un{z}^{\un{g}} = w$.} of $\un{z}$ such that $W_J w = W_J \un{x}^{\un{e}} = W_J z$. Since $z$ is a minimal coset rep, this implies $z \leq w$. But since $\un{z}$ is reduced, we also have $w \leq z$, so that $w = z$. So the first two maps $Q_{\un{e}} \hookrightarrow \un{x} \rightarrow \un{z}$ land in $Q_z$, so that projecting onto $Q_z$ and including back into $\un{z}$ does nothing. Thus, the maps (\ref{p}) and (\ref{project include}) are equal. 

Now, consider the first half of (\ref{project include}): $$Q_{\un{e}} \hookrightarrow \un{x} \xrightarrow{SLL_{\un{x}, \un{e}}} \un{z} \twoheadrightarrow Q_z$$ This is exactly $q_{\un{e}}$. Similarly, the second half is (again, up to a non-zero constant) the map $\overline{q_{\un{f}}}$, i.e. $q_{\un{f}}$ flipped upside down. But since $q_{\un{f}} \in Q$, we have $\overline{q_{\un{f}}} = q_{\un{f}}$. Thus, the map in (\ref{project include}) is $c \cdot q_{\un{e}} \cdot q_{\un{f}}$ for some non-zero $c \in Q$, and so $p_{e, f}^{e, f} = c \cdot q_{\un{e}} \cdot q_{\un{f}}$. So $p_{e, f}^{e, f} \neq 0$ for any $\un{e} \subset \un{x}$, $\un{f} \subset \un{y}$. $\hfill \square$\\
\\
Thus, the localized maps $\mathbb{SDL}_{\un{f}, \un{e}} \otimes id_Q$ are upper-triangular with respect to the partial order `$\preceq$', and thus form a basis, proving Proposition \ref{localized-basis}. $\hfill \square$
\begin{cor} \label{linearly independent}
    Given $\un{x}, \un{y} \in \mathcal{M}_{BS}$, let $\mathcal{SDL}_{\un{x}, \un{y}}$ contain one double-leaf map $\mathbb{SDL}_{\un{f}, \un{e}}$ for every pair of subexpressions $\un{e} \subset \un{x}$, $\un{f} \subset \un{y}$ such that $W_J \un{x}^{\un{e}} = W_J \un{y}^{\un{f}}$. Then the maps in $\mathcal{SDL}_{\un{x}, \un{y}}$ are linearly independent. 
\end{cor}

\section{Spanning} \label{Spanning}
We will now prove that the double-leaves span. More precisely:
\begin{prop} \label{span}
    Given two expressions $\un{x}, \un{y}$, let $\mathcal{SDL}_{\un{x}, \un{y}}$ be a collection of maps containing one double-leaf $\mathbb{SDL}_{\un{e}, \un{f}}$ for every pair of subexpressions $\un{e} \subset \un{x}$, $\un{f} \subset \un{y}$ such that $W_J \un{x}^{\un{e}} = W_J \un{y}^{\un{f}}$. Then $\mathcal{SDL}_{\un{x}, \un{y}}$ spans Hom$_{\mathcal{M}_{BS}}(\un{x}, \un{y})$. 
\end{prop}
Our proof draws heavily on the proof that double-leaves span in $\mathcal{D}$ presented in Chapter 7 of \cite{SoergelCalculus}. 

First, some terminology. Let us take a diagrammatic morphism in $\mathcal{M}_{BS}$. We can arrange this diagram so that it has no horizontal lines. Then, at any particular height, this diagram has a certain \textit{width}. 
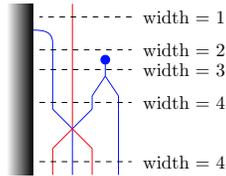
\begin{center}
    \scalebox{0.7}{\input{fig27.tikz}}
\end{center}
Then, we classify each of our generators as \textit{positive} (+), \textit{negative} (-) or \textit{neutral} (=), based on how they affect the width.
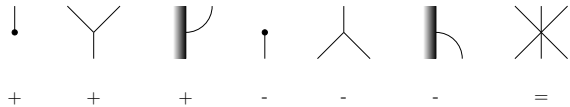
\begin{center}
    \scalebox{0.7}{\input{fig28.tikz}}
\end{center}
We don't consider polynomials to affect the width. We say that a diagram is \textit{negative-positive} if it consists of negative and neutral (non-positive) generators, followed by positive and neutral (non-negative) generators. A diagram is \textit{strictly} negative-positive if the width actually shrinks and then increases again non-trivially. Note that this notion is not well-defined up to isotopy. We say that a map is (strictly) negative-positive if it can be moved via isotopy into a diagram which is (strictly) negative-positive. 

Take a reduced expression $\un{z}$ for a minimal coset representative $z$. Let $I_{\un{z}}$ be the right ideal (top ideal, thinking diagrammatically) of morphisms which are positive on top. So, a map in $I_{\un{z}}$ (up to isotopy) finishes with one of our positive generators, followed by rex moves. But note that since $\un{z}$ is a reduced expression for a minimal coset representative, this positive generator cannot be the trivalent vertex, nor a strand coming from the wall. Therefore a morphism in $I_{\un{z}}$ must end in a down-dot \begin{tikzpicture}
    \draw[red] (0, 0) -- (0, 0.4);
    \filldraw[red] (0, 0) circle(1pt);
\end{tikzpicture}
followed by rex moves. Using the Jones-Wenzl relation (see (5.7) in \cite{SoergelCalculus}), we know that this dot will pull through our rex moves to the top. Therefore, $I_{\un{z}}$ is the right (top) ideal of morphisms with a down-dot at the top.
\subsection{Preliminaries} \label{preliminaries}
\begin{lem} \label{rex}
    Let $\un{w}$ and $\un{w}'$ be two reduced expressions for the same element, and let $\beta$ and $\beta'$ be two different rex moves from $\un{w}$ to $\un{w}'$. Then $\beta - \beta'$ has a strictly negative-positive decomposition.
\end{lem}
\noindent \textit{Proof.} Lemma 7.4 of \cite{SoergelCalculus} gives this in the case of $\mathcal{D}$, so it is also true in $\mathcal{M}_{BS}$. $\hfill \square$
\begin{lem} \label{factor2}
    Let $\un{x}$ be an arbitrary expression. Then the map $id_{\un{x}}$ is a sum of negative-positive maps which each factor through some reduced expression $\un{z}$ for a minimal coset representative $z \leq \un{x}$. 
\end{lem}

\begin{proof}
    By Lemma 7.5 in \cite{SoergelCalculus}, we can write $id_{\un{x}}$ as a sum of negative-positive maps factoring through reduced expressions. Then, we can move strands $s \in J$ to the left via rex moves (see proof of Lemma 7.5 in \cite{SoergelCalculus}), and apply the following relation:
    \begin{equation} \label{wall}
    \begin{tikzpicture}[scale = 0.7]
        \shade[left color = white, right color = black] (0, 0) rectangle (0.5, 2);
        \draw[red] (1, 0) -- (1, 2);
        \node at (1.5, 1) {=};
        \shade[left color = white, right color = black] (2, 0) rectangle (2.5, 2);
        \draw[red] (3.5, 0) -- (3.5, 2);
        \draw[red] (2.5, 1) -- (2.8, 1);
        \draw[red] (3.2, 1) -- (3.5, 1);
        \node at (4, 1) {=};
        \filldraw[red] (2.8, 1) circle(2pt);
        \filldraw[red] (3.2, 1) circle(2pt);
        \node at (4.5, 1) {$\frac{1}{2}$};
        \shade[left color = white, right color = black] (4.8, 0) rectangle (5.3, 2);
        \draw[red] (6, 0) -- (6, 2);
        \draw[red] (5.3, 1) -- (6, 1);
        \node at (5.65, 1.3) {$\al_s$};
        \node at (6.5, 1) {+};
        \node at (7, 1) {$\frac{1}{2}$};
        \shade[left color = white, right color = black] (7.3, 0) rectangle (7.8, 2);
        \draw[red] (8.5, 0) -- (8.5, 2);
        \draw[red] (7.8, 1) -- (8.5, 1);
        \node at (8.15, 0.7) {$\al_s$};
        \node at (9, 1) {=};
        \node at (9.5, 1) {$\frac{1}{2}$};
        \shade[left color = white, right color = black] (9.8, 0) rectangle (10.3, 2);
        \draw[red, rounded corners = 5] (11, 0) -- (11, 0.8) -- (10.3, 0.8);
        \draw[red, rounded corners = 5] (11, 2) -- (11, 1.2) -- (10.3, 1.2);
        \node at (10.65, 1.5) {$\al_s$};
        \node at (11.5, 1) {+};
        \node at (12, 1) {$\frac{1}{2}$};
        \shade[left color = white, right color = black] (12.3, 0) rectangle (12.8, 2);
        \draw[red, rounded corners = 5] (13.5, 0) -- (13.5, 0.8) -- (12.8, 0.8);
        \draw[red, rounded corners = 5] (13.5, 2) -- (13.5, 1.2) -- (12.8, 1.2);
        \node at (13.15, 0.5) {$\al_s$};
    \end{tikzpicture}
\end{equation} We can then apply induction on the length of $\un{x}$.
\end{proof}
We first prove spanning in the simplest case.

\begin{lem} \label{empty}
    Take an expression $\un{x}$ in $W_J$ and a map $\phi : \un{x} \rightarrow \un{\varnothing}$ in $\mathcal{M}_{BS}$. Then $\phi$ lies in the right $R$-span of light-leaves $SLL_{\un{x}, \un{e}} : \un{x} \rightarrow \un{\varnothing}$.
\end{lem}
\noindent \textit{Proof.} By isotopy, move all wall plug-ins to the top. Using polynomial forcing, move all polynomials to the right. Any up-dots can be pulled down, as up-dots pull through rex moves via the Jones-Wenzl relation. Eventually each up-dot will reach either the bottom of the diagram, a down-dot (forming a barbell), a splitting trivalent vertex (thus disappearing), or a merging trivalent vertex (forming a cap). In the case where it reaches a down-dot, forming a barbell, we can again move this to the right using polynomial forcing. If it hits a merging trivalent vertex (a cap), we apply (\ref{wall}) to slide the cap into the wall. After repeating this process and applying wall-relations, eventually the remaining diagrams will consist only of strands plugging into the wall and up-dots, which are exactly the light-leaves. $\hfill \square$

\begin{cor} \label{identity}
    Take an expression $\un{x}$ in $W_J$. Then any map $\phi: \un{x} \rightarrow \un{\varnothing}$ is a sum of maps made only of strands from $\un{x}$ into the wall, with polynomials in the gaps.
\end{cor}
\begin{proof}
    By Lemma \ref{empty}, $\phi$ is in the span of light-leaves. Any up-dots in a light-leaf can be pulled into the wall using (\ref{wall-fusion}), leaving only strands entering the wall and polynomials. 
\end{proof}

\noindent The bulk of the work in this chapter will go towards proving the following:
\begin{prop} \label{ll}
    Let $\un{z}$ be a reduced expression for a minimal coset representative $z$, and let $\un{x}$ be an arbitrary expression. Choose a light-leaf map $SLL_{\un{x}, \un{e}}$ for each $\un{e} \subset \un{x}$ such that $W_J \un{x}^{\un{e}} = W_J z$. These maps span Hom$_{\mathcal{M}_{BS}}(\un{x}, \un{z}) / I_{\un{z}}$.
\end{prop}
\noindent The analogous statement holds for $\mathcal{D}$. 
\begin{prop} \label{nll}
    Let $\un{w}$ be a reduced expression for $w \in W$, and let $\un{x}$ be an arbitrary expression. Choose a light-leaf map $LL_{\un{x}, \un{e}}$ for each $\un{e} \subset \un{x}$ such that $ \un{x}^{\un{e}} = w$. These maps span Hom$_{\mathcal{D}}(\un{x}, \un{w}) / I_{\un{w}}$.
\end{prop}
\begin{proof}
    See Proposition 7.6 of \cite{SoergelCalculus}.
\end{proof}

\noindent \textit{Proof of Proposition \ref{span} given \ref{ll}:} 
 By Lemma \ref{factor2}, we assume $\phi: \un{x} \to \un{y}$ factors through a reduced expression $\un{z}$ of an m.c.r. $z \leq \un{x}$, i.e., $\phi = f \circ g$ for $g: \un{x} \to \un{z}, \; f: \un{z} \rightarrow \un{y}$. We induct on $z$ in the Bruhat order. The base case is immediate from Proposition \ref{ll}.
Then, by Proposition \ref{ll} we have $g = g_z + g_l$, where $g_z$ is a sum of spherical light-leaves and $g_l \in I_{\un{z}}$. Dually, $f = f_z + f_l$.
The term $f_z g_z$ consists of double leaves. The term $f_z g_l$ involves a map in $I_{\un{z}}$, which factors through an expression $\un{w}$ with $l(\un{w}) < l(\un{z})$. Since $id_{\un{w}}$ factors through m.c.r.s $v < z$, the error terms reduce to the inductive hypothesis. The same reasoning applies to to other terms $f_l g_l$ and $f_l g_z$. $\hfill \square$\\
\\
\noindent It remains to prove Proposition \ref{ll}. Let $\phi : \un{x} \rightarrow \un{z}$. Moving wall strands to the top, we have:
\begin{equation} \label{uz}
    \scalebox{0.7}{\begin{tikzpicture}
        \node at (0, 1.5) {$\phi = $};
        \shade[left color = white, right color = black] (1, 0) rectangle (2, 3);
        \draw (3, 0) -- (3, 0.5);
        \draw (2.5, 0.5) -- (5.5, 0.5) -- (5.5, 1.75) -- (2.5, 1.75) -- cycle;
        \draw (5, 0) -- (5, 0.5);
        \node at (4, 0.25) {$\dots$};
        \draw[rounded corners = 5] (2.75, 1.75) -- (2.75, 2) -- (2, 2);
        \draw[rounded corners = 10] (3.05, 1.75) -- (3.05, 2.25) -- (2, 2.25);
        \node at (3.35, 2) {$\cdots$};
        \draw[rounded corners = 10] (3.65, 1.75) -- (3.65, 2.5) -- (2, 2.5);
        \draw (4.05, 1.75) -- (4.05, 3);
        \draw (4.35, 1.75) -- (4.35, 3);
        \node at (4.75, 2) {$\cdots$};
        \draw (5.15, 1.75) -- (5.15, 3);
        \draw [decorate,decoration={brace,amplitude=3pt,mirror}]
  (2.6,1.7) -- (3.75,1.7) node[midway,yshift=-1em]{$\un{u}$};
  \draw [decorate,decoration={brace,amplitude=3pt,mirror}]
  (3.95,1.7) -- (5.25,1.7) node[midway,yshift=-1em]{$\un{z}$};
  \node at (4, 1) {??};
    \end{tikzpicture}}
\end{equation}
where $\un{u}$ is an expression in $J$. By wall relations, we may assume $\un{u}$ is reduced. Since $\un{z}$ is a reduced expression (rex) of an m.c.r., $(\un{u}, \un{z})$ is reduced. By Proposition \ref{nll}, the inner box is a sum of (non-spherical) light-leaves plus maps in $I_{(\un{u}, \un{z})}$. (In fact, we get to choose which light-leaf maps to use here. This will be important, and we will describe this choice in detail in the section below). Maps in $I_{(\un{u}, \un{z})}$ have a down-dot at the top; if the dot is on $\un{u}$, it slides into the wall, reducing the length of $\un{u}$. Repeating this process reduces $\phi$ to a sum of non-spherical light-leaves (followed by plugging strands into the wall) and maps in $I_{\un{z}}$.
\subsection{Construction of non-spherical light-leaves} \label{construction2}

 We noted above that we get to choose the particular construction of these non-spherical light-leaves. We will now specify this choice. Roughly speaking, the idea is to make these non-spherical light-leaves look as much like spherical light-leaves as possible. That is, our spherical light-leaves have various strands which plug into the wall on the left. Non-spherical light-leaves don't have any wall to plug into, so we will instead collect all these ``wall-strands" on the left of the diagram. 
 
 Suppose we are constructing a light-leaf $LL_{\un{x}, \un{e}}: \un{x} \rightarrow \un{w}$ for $\un{x} = (t_1, ..., t_M)$. At step $k$ we construct a map $LL_k$ to a rex $\un{w}_k$ of $w_k := \un{x}_{\leq k}^{\un{e}_{\leq k}} = t_1^{e_1} \cdots t_k^{e_k}$. We decompose the target element $w_k = u_k z_k$ for $u_k \in W_J, \; z_k \in {}^J W$ and define the intermediate expression $\un{w}_k := (\un{u}_k, \un{z}_k)$ for rex's $\un{u}_k, \; \un{z}_k$.
The choice of rex moves in $\phi_k: (\un{u}_{k - 1}, \un{z}_{k - 1}, t_k) \rightarrow (\un{u}_k, \un{z}_k)$ depends on the comparison between the non-spherical label $d_k$ and the spherical label $d_k'$.

\textbf{Case 1: $d_k = d_k'$.}
If $d_k = d_k' = U0$ or $U1$, then we can apply separate rex moves $\gamma$ and $\al$ to the $\un{u}$'s and the $\un{z}$'s respectively:

\begin{equation}
    \scalebox{0.7}{\input{fig37.tikz}}
\end{equation}

\noindent (We include $LL_{k - 1}$ in our diagrams, so that the diagrams depict the map $LL_k$). If $d_k = d_k' = D0$ or $D1$, then we need to apply some rex move $\beta$ to move $t_k$ to the right. Since $d_k' = D$, this $\beta$ can be applied solely to $\un{z}_{k - 1}$:

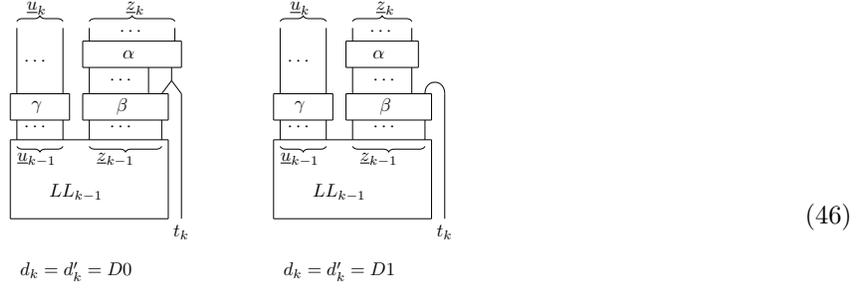
\begin{equation}
    \scalebox{0.7}{\input{fig38.tikz}}
\end{equation}

Here, $\alpha$ is some other rex move to $\un{z}_k$. The point is that if $d_k = d_k'$, the $\un{u}$ and $\un{z}$ components can be dealt with separately, so that we have something of the following form.
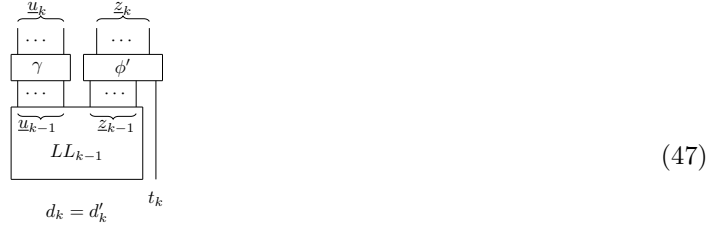
\begin{equation} \label{equal}
    \scalebox{0.7}{\input{fig7.tikz}}
\end{equation}
We can notice that in each of the above four cases, the map $\phi'$ is exactly the $\phi$ associated to $d_k'$ in the $k^{\text{th}}$ step of the spherical light-leaf construction. Thus we are making our non-spherical light-leaves look like spherical light-leaves.\\

\noindent \textbf{Case 2: $d_k \neq d_k'$.}
This occurs when $d_k' \in \{X0, X1\}$. We treat the sub-cases as follows:

\begin{itemize}
    \item $d_k \in \{U0, U1\}$:
\end{itemize} 
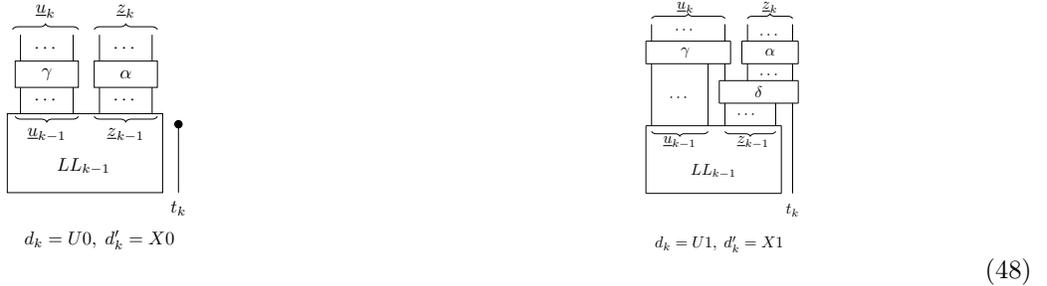
\begin{equation} \label{U0U1}
\begin{minipage}{0.48\textwidth}
    \centering
    \scalebox{0.7}{\begin{tikzpicture}
        \draw (1.25, 0) -- (4.2, 0) -- (4.2, 1.5) -- (1.25, 1.5) -- cycle;
        \node at (2.7, 0.5) {$LL_{k - 1}$};
        \draw (1.5, 1.5) -- (1.5, 2);
        \draw (1.5, 2.5) -- (1.5, 3);
        \draw (1.4, 2) -- (2.6, 2) -- (2.6, 2.5) -- (1.4, 2.5) -- cycle;
        \node at (2, 2.25) {$\gamma$};
        \draw (2.9, 2) -- (4.1, 2) -- (4.1, 2.5) -- (2.9, 2.5) -- cycle;
        \node at (3.5, 2.25) {$\alpha$};
        \node at (2, 1.75) {$\cdots$};
        \draw (2.5, 1.5) -- (2.5, 2);
        \draw (2.5, 2.5) -- (2.5, 3);
        \draw (3, 1.5) -- (3, 2);
        \draw (3, 2.5) -- (3, 3);
        \node at (3.5, 1.75) {$\cdots$};
        \node at (2, 2.75) {$\cdots$};
        \node at (3.5, 2.75) {$\cdots$};
        \draw (4, 1.5) -- (4, 2);
        \draw (4, 2.5) -- (4, 3);
        \draw (4.5, 0) -- (4.5, 1.3);
        \filldraw (4.5, 1.3) circle(2pt);
        \node at (4.5, -0.3) {$t_k$};
        \node at (3, -0.9) {$d_k = U0, \; d_k' = X0$};
        \draw [decorate,decoration={brace,amplitude=3pt,mirror}]
  (1.4,1.45) -- (2.6,1.45) node[midway,yshift=-1em]{$\un{u}_{k - 1}$};
  \draw [decorate,decoration={brace,amplitude=3pt,mirror}]
  (2.9,1.45) -- (4.1,1.45) node[midway,yshift=-1em]{$\un{z}_{k - 1}$};
  \draw [decorate,decoration={brace,amplitude=3pt}]
  (1.35,3.1) -- (2.6,3.1) node[midway,yshift=1em]{$\un{u}_k$};
  \draw [decorate,decoration={brace,amplitude=3pt}]
  (2.9,3.1) -- (4.1,3.1) node[midway,yshift=1em]{$\un{z}_k$};
    \end{tikzpicture}}
\end{minipage}
\quad
\begin{minipage}{0.48\textwidth}
    \centering
    \scalebox{0.6}{\input{fig39.tikz}}
\end{minipage}
\end{equation}

 For $d_k=U0, d_k'=X0$, the picture is identical to $d_k = d_k' = U0$. If $d_k' = X1$, then by Lemma \ref{wall-crossing} (b) there is some $s \in J$ such that $z_{k - 1}t_k = s z_{k - 1}$. By Lemma \ref{wall-crossing} (a), $(\un{z}_{k - 1}, t_k)$ is reduced so that there is some rex move $\delta : (\un{z}_{k - 1}, t_k) \rightarrow (s, \un{z}_{k - 1})$. Then, since $d_k = U1$, $(\un{u}_{k - 1}, s)$ must also be reduced, so there is some rex move $\gamma : (\un{u}_{k - 1}, s) \rightarrow \un{u}_k$.

\begin{itemize}
    \item $d_k = D1, d_k' = X1$:
\end{itemize} 
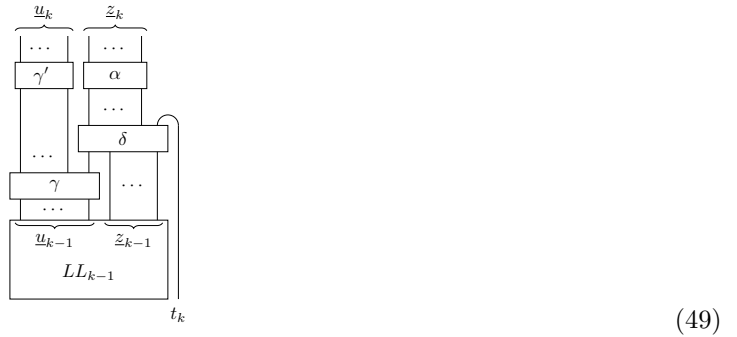
\begin{equation} \label{D1X1}
    \scalebox{0.7}{\begin{tikzpicture}
        \draw (0, 0) -- (3, 0) -- (3, 1.5) -- (0, 1.5) -- cycle;
        \node at (1.5, 0.5) {$LL_{k - 1}$};
        \draw (0, 1.9) -- (1.7, 1.9) -- (1.7, 2.4) -- (0, 2.4) -- cycle;
        \node at (0.85, 2.15) {$\gamma$};
        \draw (1.3, 2.8) -- (3, 2.8) -- (3, 3.3) -- (1.3, 3.3) -- cycle;
        \node at (2.15, 3.05) {$\delta$};
        \draw (0.2, 1.5) -- (0.2, 1.9);
        \draw (0.2, 2.4) -- (0.2, 4);
        \node at (0.65, 2.7) {$\cdots$};
        \draw (1.1, 2.4) -- (1.1, 4);
        \node at (0.85, 1.7) {$\cdots$};
        \draw (1.5, 1.5) -- (1.5, 1.9);
        \draw (1.5, 3.3) -- (1.5, 4);
        \draw (1.5, 4.5) -- (1.5, 5);
        \draw (1.5, 2.4) -- (1.5, 2.8);
        \draw (1.9, 1.5) -- (1.9, 2.8);
        \draw (2.8, 1.5) -- (2.8, 2.8);
        \draw (2.5, 3.3) -- (2.5, 4);
        \draw (2.5, 4.5) -- (2.5, 5);
        \node at (2, 4.75) {$\cdots$};
        \node at (2, 3.6) {$\cdots$};
        \node at (2.35, 2.15) {$\cdots$};
        \draw (3.2, 0) -- (3.2, 3.3);
        \node at (3.2, -0.3) {$t_k$};
        \draw (3.2, 3.3) arc[start angle = 0, end angle = 180, radius = 0.2];
        \draw (0.1, 4) -- (1.2, 4) -- (1.2, 4.5) -- (0.1, 4.5) -- cycle;
        \node at (0.6, 4.25) {$\gamma'$};
        \draw (0.2, 4.5) -- (0.2, 5);
        \draw (1.1, 4.5) -- (1.1, 5);
        \node at (0.6, 4.75) {$\cdots$};
        \draw (1.4, 4) -- (2.6, 4) -- (2.6, 4.5) -- (1.4, 4.5) -- cycle;
        \node at (2, 4.25) {$\alpha$};
        \draw [decorate,decoration={brace,amplitude=3pt,mirror}]
  (0.1,1.45) -- (1.6,1.45) node[midway,yshift=-1em]{$\un{u}_{k - 1}$};
  \draw [decorate,decoration={brace,amplitude=3pt,mirror}]
  (1.8,1.45) -- (2.9,1.45) node[midway,yshift=-1em]{$\un{z}_{k - 1}$};
   \draw [decorate,decoration={brace,amplitude=3pt}]
  (0.1,5.1) -- (1.2,5.1) node[midway,yshift=1em]{$\un{u}_{k}$};
   \draw [decorate,decoration={brace,amplitude=3pt}]
  (1.4,5.1) -- (2.6,5.1) node[midway,yshift=1em]{$\un{z}_{k}$};
    \end{tikzpicture}}
\end{equation}
Since $d_k = D1$, we need to apply some rex move $\beta$ to $(\un{u}_{k - 1}, \un{z}_{k - 1})$ to get $t_k$ to the right of the expression. This is done as follows: Let $\un{u}_{k - 1} = (s_1, ..., s_l)$, $\un{z}_{k - 1} = (r_1, ..., r_m)$. Since $d_k = D1$, we have $u_{k - 1} z_{k - 1} t_k = s_1 \cdots s_l r_1 \cdots r_m t_k = s_1 \cdots \hat{s_i} \cdots s_l r_1 \cdots r_m$ for some $i$, where the hat denotes deletion (the deleted generator cannot be among the $r_i$'s by Lemma \ref{wall-crossing} (a)). Since $d_k' = X1$, there is some $s \in J$ such that $z_{k - 1} t_k = s z_{k - 1}$. This means $u_{k - 1} z_{k - 1} t_k = s_1 \cdots s_l s z_{k - 1} = s_1 \cdots \hat{s_i} \cdots s_l z_{k - 1}$, meaning $s$ is in the right descent set of $u_{k - 1}$. Therefore there is some rex of $u_{k - 1}$ with $s$ on the right (in fact it is $(s_1, ..., \hat{s_i}, ..., s_l, s)$). We apply a rex move $\gamma$ to $\un{u}_{k - 1}$ to move this $s$ to the right, and then apply a rex move $\delta : (s, \un{z}_{k - 1}) \rightarrow (\un{z}_{k - 1}, t_k)$. We have thus moved $t_k$ to the right as required. After applying a cap to this $t_k$, we apply rex moves $\gamma'$ and $\delta'$ to make sure the targets are correct.

The key in this case will be that we can rotate the $s$ and $t_k$ strands in $\delta$:
\begin{equation} \label{rotate}
    \scalebox{0.7}{\input{fig14.tikz}}
\end{equation}
Here $\tilde{\delta}$ is some other rex move. This will assist us in turning non-spherical light-leaves into spherical light-leaves.

\begin{itemize}
    \item $d_k = D0, \; d_k' = X0$:
\end{itemize}

\begin{equation} \label{D0X0}
    \scalebox{0.7}{\begin{tikzpicture}
        \draw (0, 0) -- (3, 0) -- (3, 1.5) -- (0, 1.5) -- cycle;
        \node at (1.5, 0.5) {$LL_{k - 1}$};
        \draw (0, 1.9) -- (1.7, 1.9) -- (1.7, 2.4) -- (0, 2.4) -- cycle;
        \node at (0.85, 2.15) {$\gamma$};
        \draw (0, 5) -- (1.6, 5) -- (1.6, 5.5) -- (0, 5.5) -- cycle;
        \node at (0.85, 5.25) {$\gamma'$};
        \draw (1.8, 5) -- (3.2, 5) -- (3.2, 5.5) -- (1.8, 5.5) -- cycle;
        \node at (2.5, 5.25) {$\alpha$};
        \draw (1.3, 2.8) -- (3, 2.8) -- (3, 3.3) -- (1.3, 3.3) -- cycle;
        \node at (2.15, 3.05) {$\delta$};
        \draw (0.2, 1.5) -- (0.2, 1.9);
        \draw (0.2, 2.4) -- (0.2, 5);
        
        \node at (0.65, 2.7) {$\cdots$};
        \draw (1.1, 2.4) -- (1.1, 5);
        \draw (1.1, 5.5) -- (1.1, 6);
        \node at (0.85, 1.7) {$\cdots$};
        \draw (1.5, 1.5) -- (1.5, 1.9);
        \draw (1.5, 3.3) -- (1.5, 4);
        \draw (1.5, 2.4) -- (1.5, 2.8);
        \draw (1.9, 1.5) -- (1.9, 2.8);
        \draw (2.8, 1.5) -- (2.8, 2.8);
        \node at (2.2, 3.6) {$\cdots$};
        \node at (2.35, 2.15) {$\cdots$};
        \draw (3.2, 0) -- (3.2, 3.3);
        \node at (3.2, -0.3) {$t_k$};
        \draw (3.2, 3.3) -- (3, 3.5);
        \draw (2.8, 3.3) -- (3, 3.5);
        \draw (3, 3.5) -- (3, 4);
        \draw (1.3, 4) -- (3.2, 4) -- (3.2, 4.5) -- (1.3, 4.5) -- cycle;
        \node at (2.25, 4.25) {$\delta'$};
        
        \draw (1.5, 4.5) -- (1.5, 5);
        \draw (3, 4.5) -- (3, 5);
        \draw (3, 5.5) -- (3, 6);
        \draw (1.9, 4.5) -- (1.9, 5);
        \draw (1.9, 5.5) -- (1.9, 6);
        \node at (2.45, 4.75) {$\cdots$};
        \node at (2.45, 5.75) {$\cdots$};
        \node at (0.65, 5.75) {$\cdots$};
        \draw (0.2, 5.5) -- (0.2, 6);
        \draw (1.5, 5.5) -- (1.5, 6);
        \draw [decorate,decoration={brace,amplitude=3pt,mirror}]
  (0.1,1.45) -- (1.6,1.45) node[midway,yshift=-1em]{$\un{u}_{k - 1}$};
  \draw [decorate,decoration={brace,amplitude=3pt,mirror}]
  (1.8,1.45) -- (2.9,1.45) node[midway,yshift=-1em]{$\un{z}_{k - 1}$};
  \draw [decorate,decoration={brace,amplitude=3pt}]
  (0.1,6.1) -- (1.6,6.1) node[midway,yshift=1em]{$\un{u}_k$};
  \draw [decorate,decoration={brace,amplitude=3pt}]
  (1.8,6.1) -- (3.1,6.1) node[midway,yshift=1em]{$\un{z}_k$};
    \end{tikzpicture}}
\end{equation}
This final case is the most daunting. It is in fact similar to the previous case of $d_k = D1$, $d_k' = X1$: we first apply $\gamma$ as $\delta$ in almost the same way. The difference will be that $\delta$ will be a carefully chosen rex move $(s, \un{z}_{k - 1}) \rightarrow (\un{\tilde{z}}_{k - 1}, t_k)$, where $\un{z}_{k - 1}$ and $\un{\tilde{z}}_{k - 1}$ are not necessarily the same. We then apply a third rex move $\delta': (\un{\tilde{z}}_{k - 1}, t_k) \rightarrow (s, \un{\tilde{z}}_{k - 1})$. Finally we apply rex moves $\gamma'$ and $\alpha$ to ensure our map has the correct target $(\un{u}_k, \un{z}_k)$.

We want to choose $\delta$ and $\delta'$ so that we can repeatedly apply two-colour associativity\footnote{Equation 5.6 in \cite{SoergelCalculus}} in order to pull the trivalent vertex from the right to the left (This will be necessary when we turn our light-leaves into spherical light-leaves). This is best illustrated through examples. If $\delta$ is a single $2m$-valent vertex, then we can set $\delta' = \overline{\dl}$ (i.e. $\delta$ upside-down), and we can pull the trivalent vertex to the left using a single application of two-colour associativity:
\begin{center}
    \scalebox{0.7}{\begin{tikzpicture}
        \draw[red] (0.5, 0.5) -- (1, 1) -- (1, 2.5) -- (0.5, 3);
        \draw[red] (1.5, 0.5) -- (1, 1);
        \draw[red] (1, 2.5) -- (1.5, 3);
        \draw[blue] (1, 0.5) -- (1, 1);
        \draw[blue, rounded corners = 15] (1, 1) -- (1.75, 1.75) -- (1, 2.5);
        \draw[blue, rounded corners = 15] (1, 1) -- (0.25, 1.75) -- (1, 2.5);
        \draw[blue, rounded corners = 10] (2, 0.5) -- (2, 1.75) -- (1.53, 1.75);
        \draw[blue] (1, 2.5) -- (1, 3);
        \node at (2.5, 1.75) {=};
        \draw[red] (3, 0.5) -- (3, 3);
        \draw[blue] (3.5, 0.5) -- (3.5, 3);
        \draw[red] (3, 1.75) -- (3.5, 1.75) -- (4, 0.5);
        \draw[red] (3.5, 1.75) -- (4, 3);
        \draw[blue, rounded corners = 10] (4.5, 0.5) -- (4.5, 1.75) -- (3.5, 1.75);
    \end{tikzpicture}}
\end{center}
This also works if $\delta$ is a chain of $2m$-valent vertices right next to one another, moving left to right:
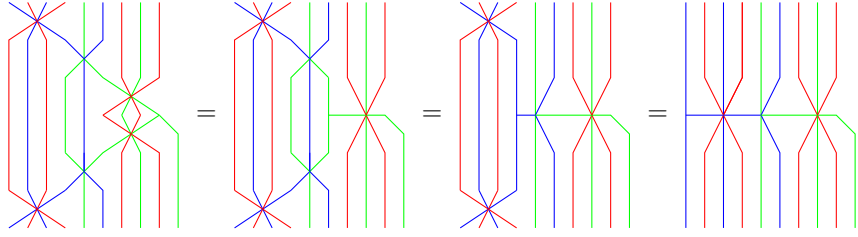
\begin{center}
    \begin{tikzpicture}[scale = 0.5]
        \draw[blue] (0, 0) -- (1.5, 1);
        \draw[red] (0.5, 0) -- (1, 1);
        \draw[blue] (1, 0) -- (0.5, 1);
        \draw[red] (1.5, 0) -- (0, 1);
        \draw[blue] (1.5, 1) -- (2, 1.5);
        \draw[green] (2, 0) -- (2, 1.5);
        \draw[blue] (2.5, 0) -- (2.5, 1) -- (2, 1.5);
        \draw[green] (1.5, 2) -- (2, 1.5) -- (2.5, 2) -- (4, 3);
        \draw[blue] (2, 1.5) -- (2, 2);
        \draw[red] (3, 0) -- (3, 2) -- (3.5, 3);
        \draw[green] (3.5, 0) -- (3.5, 2) -- (3, 3);
        \draw[red] (4, 0) -- (4, 2) -- (2.5, 3);
        \draw[green] (4, 3) -- (4.5, 2.5) -- (4.5, 0);
        \draw[green] (4, 3) -- (2.5, 4) -- (2, 4.5) -- (2, 6);
        \draw[red] (3.5, 3) -- (3, 4) -- (3, 6);
        \draw[green] (3, 3) -- (3.5, 4) -- (3.5, 6);
        \draw[red] (2.5, 3) -- (4, 4) -- (4, 6);
        \draw[blue] (2, 1.5) -- (2, 4.5) -- (2.5, 5) -- (2.5, 6);
        \draw[blue] (2, 4.5) -- (1.5, 5) -- (0, 6);
        \draw[green] (1.5, 2) -- (1.5, 4) -- (2, 4.5);
        \draw[red] (0, 1) -- (0, 5) -- (1.5, 6);
        \draw[blue] (0.5, 1) -- (0.5, 5) -- (1, 6);
        \draw[red] (1, 1) -- (1, 5) -- (0.5, 6);

        \node at (5.25, 3) {=};
        
        \draw[blue] (6, 0) -- (7.5, 1);
        \draw[red] (6.5, 0) -- (7, 1);
        \draw[blue] (7, 0) -- (6.5, 1);
        \draw[red] (7.5, 0) -- (6, 1);
        \draw[blue] (7.5, 1) -- (8, 1.5);
        \draw[green] (8, 0) -- (8, 1.5);
        \draw[blue] (8.5, 0) -- (8.5, 1) -- (8, 1.5);
        \draw[green] (7.5, 2) -- (8, 1.5) -- (8.5, 2) -- (8.5, 4);
        \draw[blue] (8, 1.5) -- (8, 2);
        \draw[red] (9, 0) -- (9, 2) -- (9.5, 3);
        \draw[green] (9.5, 0) -- (9.5, 6);
        \draw[red] (10, 0) -- (10, 2) -- (9.5, 3);
        \draw[green] (9.5, 3) -- (10, 3) -- (10.5, 2.5) -- (10.5, 0);
        \draw[green] (8.5, 4) -- (8, 4.5) -- (8, 6);
        \draw[red] (9.5, 3) -- (9, 4) -- (9, 6);
        \draw[green] (9.5, 3) -- (8.5, 3);
        \draw[red] (9.5, 3) -- (10, 4) -- (10, 6);
        \draw[blue] (8, 1.5) -- (8, 4.5) -- (8.5, 5) -- (8.5, 6);
        \draw[blue] (8, 4.5) -- (7.5, 5) -- (6, 6);
        \draw[green] (7.5, 2) -- (7.5, 4) -- (8, 4.5);
        \draw[red] (6, 1) -- (6, 5) -- (7.5, 6);
        \draw[blue] (6.5, 1) -- (6.5, 5) -- (7, 6);
        \draw[red] (7, 1) -- (7, 5) -- (6.5, 6);
        
        \node at (11.25, 3) {=};

        \draw[blue] (12, 0) -- (13.5, 1);
        \draw[red] (12.5, 0) -- (13, 1);
        \draw[blue] (13, 0) -- (12.5, 1);
        \draw[red] (13.5, 0) -- (12, 1);
        \draw[blue] (13.5, 1) -- (13.5, 5) -- (12, 6);
        \draw[blue] (13.5, 3) -- (14, 3);
        \draw[green] (14, 0) -- (14, 6);
        \draw[blue] (14.5, 0) -- (14.5, 2) -- (14, 3);
        \draw[blue] (14, 3) -- (14.5, 4) -- (14.5, 6);
        \draw[red] (15, 0) -- (15, 2) -- (15.5, 3);
        \draw[green] (15.5, 0) -- (15.5, 6);
        \draw[red] (16, 0) -- (16, 2) -- (15.5, 3);
        \draw[green] (15.5, 3) -- (16, 3) -- (16.5, 2.5) -- (16.5, 0);
        
        \draw[red] (15.5, 3) -- (15, 4) -- (15, 6);
        \draw[green] (15.5, 3) -- (14, 3);
        \draw[red] (15.5, 3) -- (16, 4) -- (16, 6);
        \draw[red] (12, 1) -- (12, 5) -- (13.5, 6);
        \draw[blue] (12.5, 1) -- (12.5, 5) -- (13, 6);
        \draw[red] (13, 1) -- (13, 5) -- (12.5, 6);

        \node at (17.25, 3) {=};

        \draw[blue] (18, 0) -- (18, 6);
        \draw[red] (18.5, 0) -- (18.5, 2) -- (19.5, 4) -- (19.5, 6);
        \draw[blue] (19, 0) -- (19, 6);
        \draw[red] (19, 3) -- (19.5, 4) -- (19.5, 6);
        \draw[red] (19, 3) -- (18.5, 4) -- (18.5, 6);
        \draw[blue] (19, 3) -- (18, 3);
        \draw[red] (19.5, 0) -- (19.5, 2) -- (19, 3);
        \draw[blue] (19, 3) -- (20, 3);
        \draw[green] (20, 0) -- (20, 6);
        \draw[blue] (20.5, 0) -- (20.5, 2) -- (20, 3);
        \draw[blue] (20, 3) -- (20.5, 4) -- (20.5, 6);
        \draw[red] (21, 0) -- (21, 2) -- (21.5, 3);
        \draw[green] (21.5, 0) -- (21.5, 6);
        \draw[red] (22, 0) -- (22, 2) -- (21.5, 3);
        \draw[green] (21.5, 3) -- (22, 3) -- (22.5, 2.5) -- (22.5, 0);
        
        \draw[red] (21.5, 3) -- (21, 4) -- (21, 6);
        \draw[green] (21.5, 3) -- (20, 3);
        \draw[red] (21.5, 3) -- (22, 4) -- (22, 6);
        
    \end{tikzpicture}
\end{center}
This final result we can arrange as 
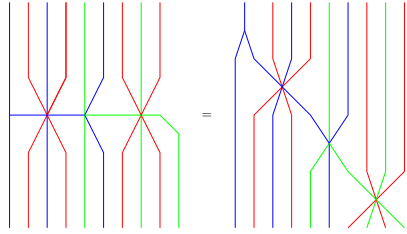
\begin{center} 
    \scalebox{0.5}{\begin{tikzpicture}
        \draw[blue] (18, 0) -- (18, 6);
        \draw[red] (18.5, 0) -- (18.5, 2) -- (19.5, 4) -- (19.5, 6);
        \draw[blue] (19, 0) -- (19, 6);
        \draw[red] (19, 3) -- (19.5, 4) -- (19.5, 6);
        \draw[red] (19, 3) -- (18.5, 4) -- (18.5, 6);
        \draw[blue] (19, 3) -- (18, 3);
        \draw[red] (19.5, 0) -- (19.5, 2) -- (19, 3);
        \draw[blue] (19, 3) -- (20, 3);
        \draw[green] (20, 0) -- (20, 6);
        \draw[blue] (20.5, 0) -- (20.5, 2) -- (20, 3);
        \draw[blue] (20, 3) -- (20.5, 4) -- (20.5, 6);
        \draw[red] (21, 0) -- (21, 2) -- (21.5, 3);
        \draw[green] (21.5, 0) -- (21.5, 6);
        \draw[red] (22, 0) -- (22, 2) -- (21.5, 3);
        \draw[green] (21.5, 3) -- (22, 3) -- (22.5, 2.5) -- (22.5, 0);
        \draw[red] (21.5, 3) -- (21, 4) -- (21, 6);
        \draw[green] (21.5, 3) -- (20, 3);
        \draw[red] (21.5, 3) -- (22, 4) -- (22, 6);
        
        \node at (23.25, 3) {=};
        
        \draw[blue] (24, 0) -- (24, 4.5) -- (24.25, 5.25) -- (24.25, 6);
        \draw[red] (24.5, 0) -- (24.5, 3) -- (26, 4.5) -- (26, 6);
        \draw[blue] (25, 0) -- (25, 3) -- (25.5, 4.5) -- (25.5, 6);
        \draw[red] (25.5, 0) -- (25.5, 3) -- (25, 4.5) -- (25, 6);
        \draw[green] (26, 0) -- (26, 1.5) -- (26.5, 2.25) -- (26.5, 6);
        \draw[blue] (26.5, 0) -- (26.5, 2.25) -- (26, 3) -- (25.25, 3.75) -- (24.5, 4.5) -- (24.25, 5.25);
        \draw[blue] (26.5, 2.25) -- (27, 3) -- (27, 6);
        \draw[red] (27, 0) -- (28.5, 1.5) -- (28.5, 6);
        \draw[green] (27.5, 0) -- (28, 1.5) -- (28, 6);
        \draw[red] (28, 0) -- (27.5, 1.5) -- (27.5, 6);
        \draw[green] (28.5, 0) -- (27, 1.5) -- (26.5, 2.25);
    \end{tikzpicture}}
\end{center}
This is just $\dl' = \overline{\delta}$ with a merging trivalent at the top left.

Of course, not all rex moves consist of a chain of $2m$-valent vertices going left to right in this way. But it turns out (and we will show) that in our situation we can always choose a $\delta$ that \textit{ends} in such a chain of $2m$-valent vertices. This will be enough for us to pull the trivalent vertex from right to left.

To make this precise, we define a \textit{sweep} to be a sequence of applications of the braid relation to an expression, moving left to right: the first one acting on letters $1, ..., m_1$, the second to letters $m_1, ..., m_2$, and so on, right to the end of the expression. For example, let $m_{st} = 3$, $m_{su} = 2$ and $m_{tu} = 3$. Here is a sweep from $(t, s, t, u, t, s)$ to $(s, t, u, t, s, t)$:
$$(t, s, t, u, t, s) \mapsto (s, t, s, u, t, s) \mapsto (s, t, u, s, t, s) \mapsto (s, t, u, t, s, t)$$

We will abuse terminology by using `sweep' to refer to the corresponding rex move in $\mathcal{D}$ or $\mathcal{M}_{BS}$, i.e. a map consisting of $2m$-valent vertices moving left to right. We saw above that if $\delta$ is a sweep, then we can apply two-colour associativity to get the following.
\begin{equation} \label{assoc}
    \scalebox{0.7}{\input{fig2.tikz}}
\end{equation}
\begin{lem}
    Take some $z \in W$ and $s, t \in S$ such that $sz = zt>z$. Then there is some rex $\un{\tilde{z}}$ of $z$ such that $(s, \un{\tilde{z}})$ can be transformed into $(\un{\tilde{z}}, t)$ via a sweep. 
\end{lem}
\noindent \textit{Proof.} We proceed by induction on $l(z)$. For the base case, if $l(z) = 0$, then $s = t$ and the statement is trivial. Suppose the statement is true when $l(z) < M$, and suppose $l(z) = M$. Take any rex $\un{z}$ of $z$. Note that since $zt > z$, we know that $\un{z}$ does not end in $t$, so neither does $(s, \un{z})$. But $(\un{z}, t)$ does end in $t$, and there is some rex move from $(s, \un{z})$ to $(\un{z}, t)$. Therefore at some step in this rex move, the last letter is changed to a $t$ from (say) a $u$. This can only happen via an application of the braid relation. That is, at this step in the rex move we have $(s_1, ..., s_n, u, t, u) \mapsto (s_1, ..., s_n, t, u, t)$ for some $s_i \in S$ (here we let $m_{tu} = 3$ for clarity of notation, but everything works the same for different $m_{tu}$). Then, we have $zt = s_1 \cdots s_n t u t$, so that $z = s_1 \cdots s_n t u$. Thus $sz = s s_1 \cdots s_n t u$. But $sz = zt = s_1 \cdots s_n u t u$, so that $s s_1 \cdots s_n = s_1 \cdots s_n u$. We can thus apply our inductive hypothesis to $z' := s_1 \cdots s_n$, with $sz' = z'u>z$.\footnote{We get $z'u>z$ by the fact that $(s_1, ..., s_n, u, t, u)$ is reduced, so that $(s_1, ..., s_n, u)$ is a rex of $z'u$.} That is, there is some rex $\un{z}'$ of $z'$ such that $(s, \un{z}')$ can be turned into $(\un{z}', u)$ by a sweep. Then set $\un{\tilde{z}} := (\un{z}', t, u)$. We can turn $(s, \un{\tilde{z}})$ into $(\un{\tilde{z}}, t)$ via a sweep: $$ (s, \un{\tilde{z}}) = (s, \un{z}', t, u) \xmapsto{sweep} (\un{z}', u, t, u) \mapsto (\un{z}', t, u, t) = (\un{\tilde{z}}, t).$$ $\hfill \square$\\
\\
This tells us how we should construct our $\delta$. Recall (see (\ref{D0X0})) that we are applying $\delta$ to the expression $(s, \un{z}_{k - 1})$, where $sz_{k - 1} = z_{k - 1} t_k$. By the above lemma, there is some expression $\un{\tilde{z}}_{k - 1}$ of $z_{k - 1}$ such that $(s, \un{\tilde{z}}_{k - 1})$ can be transformed into $(\un{\tilde{z}}_{k - 1}, t_k)$ via a sweep. First, apply a rex move $\delta_1$ from $\un{z}_{k - 1}$ to $\un{\tilde{z}}_{k - 1}$. Then apply a sweep $\delta_2$ from $(s, \un{\tilde{z}}_{k - 1})$ to $(\un{\tilde{z}}_{k - 1}, t_k)$:
\begin{equation}
    \scalebox{0.7}{\input{fig1.tikz}}
\end{equation}
We then choose $\delta' := \overline{\dl_2}$. Applying this choice to figure (\ref{D0X0}) we get 
 \begin{equation} \label{pull}
     \scalebox{0.7}{\input{fig3.tikz}}
 \end{equation}
 where we have applied (\ref{assoc}) to $\delta_2$ and $\overline{\delta_2}$.
 The final diagram on the right of (\ref{pull}) will hopefully give some sense of why we have gone to so much trouble here. Once again, we are trying to make our non-spherical light-leaves look like spherical light-leaves. If we imagine the strands joining $\gamma$ and $\gamma'$ as a `wall', then this diagram looks a lot like the step of the spherical light-leaves algorithm corresponding to $d_k' = X1$ (see diagram (\ref{X1 algorithm})).

\subsection{Proof of spanning}

Again, let $\un{z}$ be a reduced expression for some $z \in {}^J W$, and let $\un{x}$ be an arbitrary expression. We saw that any map $\phi :\un{x} \rightarrow \un{z}$ in $\mathcal{M}_{BS}$ can be considered a map $\psi : \un{x} \rightarrow (\un{u}, \un{z})$ in $\mathcal{D}$ (for some rex $\un{u}$ in $W_J$), followed by plugging the $\un{u}$-strands into the wall (see diagram (\ref{uz})). We then saw that this is a sum of our carefully chosen (non-spherical) light-leaf maps, followed by plugging the $\un{u}$-strands into the wall, as well as maps in $I_{\un{z}}$. What we would like to show is the following: A light-leaf map to $(\un{u}, \un{z})$, followed by plugging the $\un{u}$-strands into the wall, is in the span of our choice of spherical light-leaves and $I_{\un{z}}$. This would prove Proposition \ref{ll}, that any choice of spherical light-leaves spans Hom$_{\mathcal{M}_{BS}}(\un{x}, \un{z})/I_{\un{z}}$. This is eventually what we will show, but first we will show something weaker. Let $X_{\un{x}, \un{z}}$ be the collection of \textit{any} constructions of spherical light-leaves (SLL's) from $\un{x}$ to $\un{z}$, with polynomials $f \in R$ anywhere along the wall. Here is an example:
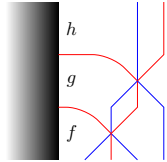
\begin{center}
    \scalebox{0.7}{\begin{tikzpicture}
        \shade[left color = white, right color = black] (0, 0) rectangle (1, 3);
        \draw[blue] (1.5, 0) -- (2, 0.5);
        \draw[blue] (2, 0.5) -- (2.5, 0);
        \draw[blue] (2, 0.5) -- (2, 1) -- (2.5, 1.5);
        \draw[red] (2, 0) -- (2, 0.5);
        \draw[red, rounded corners = 10] (2, 0.5) -- (1.5, 1) -- (1, 1);
        \draw[red] (2, 0.5) -- (2.5, 1) -- (2.5, 1.5);
        \draw[blue] (3, 0) -- (3, 1) -- (2.5, 1.5);
        \draw[red, rounded corners = 10] (2.5, 1.5) -- (2, 2) -- (1, 2);
        \draw[blue] (2.5, 1.5) -- (2.5, 3);
        \draw[red] (2.5, 1.5) -- (3, 2) -- (3, 3); 
        \node at (1.25, 0.5) {$f$};
        \node at (1.25, 1.5) {$g$};
        \node at (1.25, 2.5) {$h$};
    \end{tikzpicture}}
\end{center}

\begin{lem} \label{X}
    Let $\un{z}$ be a rex of $z \in {}^JW$, and let $\un{u}$ be a reduced expression in $W_J$. Let $\un{x}$ be an arbitrary expression. Take a (non-spherical) light-leaf $LL_{\un{x}, \un{e}}: \un{x} \rightarrow (\un{u}, \un{z})$, constructed in the way described in Section \ref{construction2}. Then, plug the top $\un{u}$-strands into the wall, and place some polynomials $f \in R$ anywhere along the wall. Then this new map is a sum of maps in $X_{\un{x}, \un{z}}$.
\end{lem}
\noindent Here is an example of the kind of map being described:
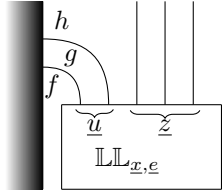
\begin{center}
    \input{fig4.tikz}
\end{center}
To avoid cluttering the diagrams, we will represent polynomials as black circles:
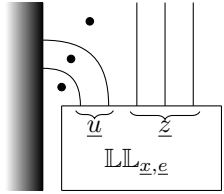
\begin{center}
    \input{fig11.tikz} 
\end{center}

We will also use a single diagram to represent sums of a certain diagram with different polynomials. For example, the sum
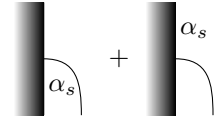
\begin{center}
    \input{fig40.tikz}
\end{center}
would be represented as
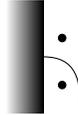
\begin{center}
    \input{fig41.tikz}
\end{center}
\noindent \textit{Proof.}
We proceed by induction on the length of $\un{x}$. If $x = \varnothing$, then the statement is trivial. Then, suppose the statement is true for $l(\un{x}) < M$, and  suppose $\un{x} = (t_1, ..., t_M)$ is of length $M$. Take some light-leaf $LL_{\un{x}, \un{e}} : \un{x} \rightarrow (\un{u}, \un{z})$, and plug the $\un{u}$-strands into the wall and place polynomials along the wall. Consider $LL_{\un{x}, \un{e}}$ as the map $LL_{M - 1}$ followed by $\phi_M$. Let $LL_{M - 1}$ have target $(\un{u}_{M - 1}, \un{z}_{M - 1})$, where $\un{u}_{M - 1}$ is a rex of $u_{M - 1} \in W_J$, and $\un{z}_{M - 1}$ is a rex of $z_{M - 1} \in \jw$ (see Section \ref{construction2}).
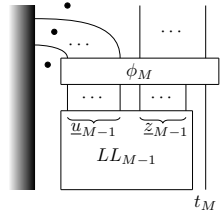
\begin{center}
   \scalebox{0.7}{\input{fig6.tikz}}
\end{center}
Once again, the subexpression $\un{e} \subset \un{x}$ has two different labellings $d_1, ..., d_M$ and $d_1', ..., d_M'$ used in the non-spherical and spherical light-leaves construction respectively. The case where $d_M = d_M'$ is easiest. Here, we can take $\phi_M$ to act separately on $\un{u}_{M - 1}$ and $\un{z}_{M - 1}$ as in figure \ref{equal}, yielding
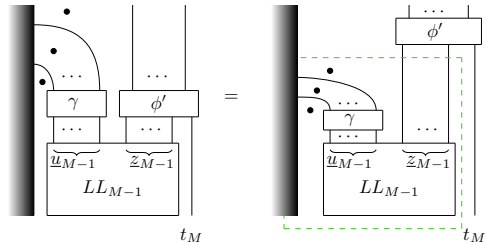
\begin{center}
    \scalebox{0.7}{\input{fig8.tikz}}
\end{center}
We can include $\gamma$ as part of $\phi_{M - 1}$ in $LL_{M - 1}$, so that by induction, the diagram in the green box above is in the span of SLL's with polynomials along the wall. We noted below figure (\ref{equal}) that the $\phi'$ here is exactly the $\phi$ associated to $d_k'$ in the $k^{\text{th}}$ step of $SLL_{\un{x}, \un{e}}$. Thus the whole diagram is in the span of SLL's with polynomials along the wall.  

The case where $d_k \neq d_k'$ is trickier, and this is where all the work we did in Section \ref{construction2} will pay off. As in that section, we break it down into the four cases:

\begin{itemize}
    \item $d_k = U1, \; d_k' = X1$:
\end{itemize}
We apply the choice of $\phi_M$ that we chose in figure (\ref{U0U1}):
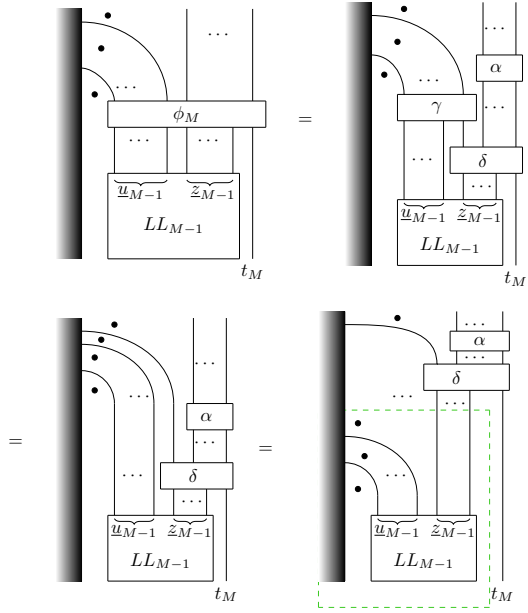
\begin{center}
    \scalebox{0.7}{\input{fig10.tikz}}
\end{center}
In the second equality we have applied Corollary \ref{identity}. Once again by induction, the map in the green box is spanned by SLL's with polynomials. The part of the diagram outside the green box is exactly the $\phi$ associated with $d_k' = X1$ (see Section \ref{spherical light-leaves}).

\begin{itemize}
    \item $d_k = U0, \;  d_k' = X0$:
\end{itemize}
We apply our choice of $\phi_M$ as in figure (\ref{U0U1}):
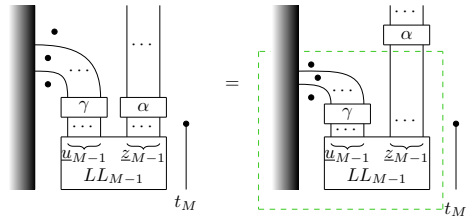
\begin{center}
    \scalebox{0.7}{\input{fig42.tikz}}
\end{center}
We can include $\gamma$ as part of $LL_{M - 1}$, and thus by induction the green box is spanned by SLL's with polynomials. The part of the map outside the green box is the $\phi$ associated with $d_k' = X0$.

\begin{itemize}
    \item $d_k = D1, \; d_k' = X1$:
\end{itemize}
Applying our choice of $\phi_M$ in figure (\ref{D1X1}), we get
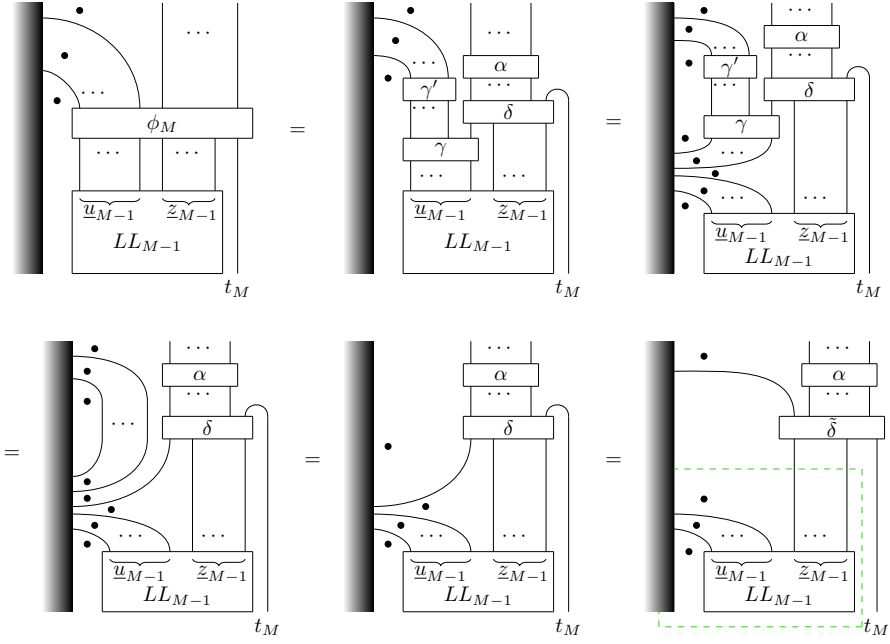
\begin{center}
    \scalebox{0.8}{\input{fig13.tikz}}
\end{center}
In the second equality, we have repeatedly applied (\ref{wall}). In the third, we have used Corollary \ref{identity} on $\gamma$ and $\gamma'$. In the fourth, we repeatedly apply (\ref{wall-bubble}). Finally in the fifth equality, we rotate $\delta$ into $\tilde{\delta}$ as described in figure (\ref{rotate}). Then, as in the other cases, by induction the map in the green box is in the span of SLLs with polynomials. The part of the map outside the green box is the $\phi$ associated to $d_k' = X1$. 
\begin{itemize}
    \item $d_k = D0, \; d_k' = X0$:
\end{itemize}
We choose $\phi_M$ as in figure \ref{D0X0}, and apply equation (\ref{pull}):
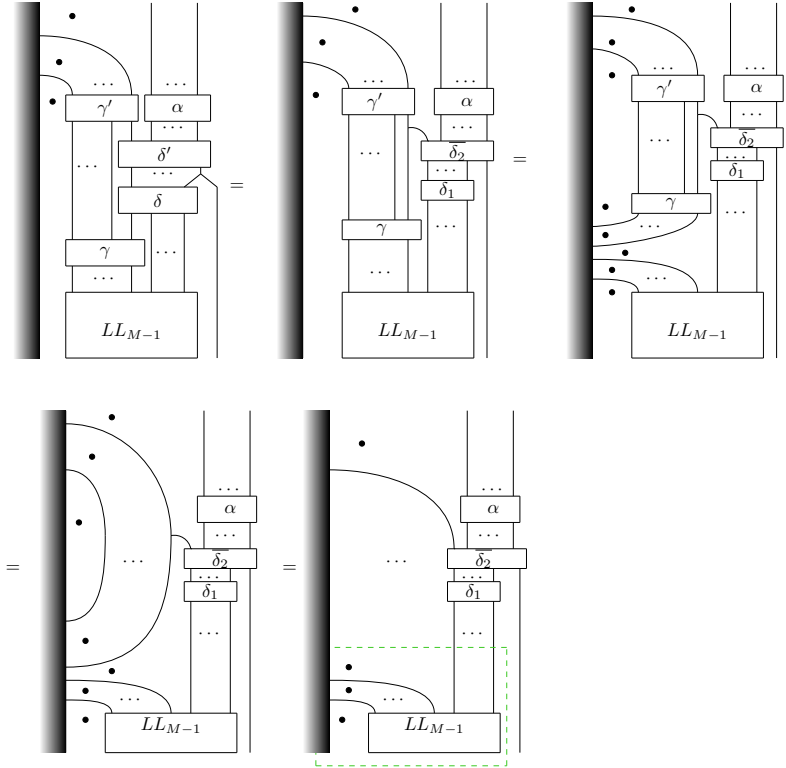
\begin{center}
    \scalebox{0.7}{\input{fig15.tikz}} 
\end{center}
The first equality is from equation (\ref{pull}). The others follow very much like the above case of $d_k = D1, \;  d_k = X1$. By induction the map in the green box is in the span of SLL's with polynomials. The part of the map outside the green box is the map $\phi$ associated with $X1$. (Note that it's okay that this is $X1$ and not $X0$: we just need the result to be \textit{some} SLL, not necessarily the ones associated to $\un{e}$.) $\hfill \square$\\
\\
We have now shown that, for a rex $\un{z}$ of $z$ an m.c.r., and arbitrary $\un{x}$, any map $\phi : \un{x} \rightarrow \un{z}$ in $\mathcal{M}_{BS}$ is in the span of \textit{some} SLL's with polynomials along the wall, and maps in $I_{\un{z}}$. If we want to prove Proposition \ref{ll}, we are left with two issues to address. First, we want to show that any map is in the span of any \textit{particular} choice of SLL constructions. Second, we need to somehow move these polynomials to the right. Our strategy will in fact be to move the polynomials up along the wall to the top-left corner of each diagram, and from there along the top of the diagram to the right.

The following lemma replaces the SLL's with any particular choice of SLL's, and moves polynomials up along the wall to the top-left corner. All that will be left after this is to move these polynomials along the top to the right.

\begin{lem}
    Given an arbitrary expression $\un{x}$, and $\un{z}$ a rex of some $z \in \jw$, let $\mathcal{S}_{\un{x}, \un{z}}$ be a collection of maps $\un{x} \rightarrow \un{z}$, containing one spherical light-leaf map $SLL_{\un{x}, \un{e}}$ for each subexpression $\un{e} \subset \un{x}$ such that $W_J \un{x}^{\un{e}} = W_J z$. Then any map $\phi: \un{x} \rightarrow \un{z}$ is a sum of maps in $\mathcal{S}_{\un{x}, \un{z}}$ with polynomials in the top-left corner, plus maps in $I_{\un{z}}$.
\end{lem}

\noindent \textit{Proof.} We prove the statement by induction on $l(\un{x})$. Suppose the statement is true for any $\un{x}$ with $l(\un{x}) < M$. Take some $\un{x}$ with $l(\un{x}) = M$, and a map $\phi : \un{x} \rightarrow \un{z}$ (for $\un{z}$ a rex of some $z \in \jw$). We know by Lemma \ref{X} (and the discussion at the end of Section \ref{preliminaries}) that $\phi$ is a sum of \textit{some} SLL maps (not necessarily the ones in $\mathcal{S}_{\un{x}, \un{z}}$) with polynomials along the wall, plus maps in $I_{\un{z}}$. Take one of these SLL maps with polynomials along the wall. It corresponds to some $\un{e} = (e_1, ..., e_M) \subset \un{x}$; in particular, $e_M$ will determine the map $\phi_M$. This SLL map can be pictured as
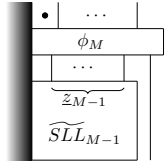
\begin{center}
    \scalebox{0.7}{\input{fig16.tikz}}
\end{center}
where $\widetilde{SLL}_{M - 1}$ is an SLL map with polynomials along the wall. Say this $\widetilde{SLL}_{M - 1}$ has target $\un{z}_{M - 1}$, a rex of $z_{M - 1} \in \jw$. Now we apply the inductive hypothesis to this $\widetilde{SLL}_{M - 1}$; that is, we can replace it with SLL's \textit{of our choice}, with polynomials in the top left corner, plus maps in $I_{\un{z}_{M - 1}}$. More precisely, for each subexpression $\un{e}' \subset \un{x}_{\leq M - 1}$\footnote{A reminder about notation: following Section \ref{spherical light-leaves}, if $\un{x} = (t_1, ..., t_M)$, then $\un{x}_{\leq M - 1} := (t_1, ..., t_{M - 1})$.} such that $W_J \un{x}_{\leq M - 1}^{\un{e}'} = W_J z_{M - 1}$, we choose a construction of $SLL_{\un{x}_{\leq M - 1}, \un{e}'} : \un{x}_{\leq M - 1} \rightarrow \un{z}_{M - 1}$. We choose this construction as follows: Given such an $\un{e}' \subset \un{x}_{\leq M - 1}$, we append $e_M$ to get a subexpression $\un{f} := (\un{e}', e_M) \subset \un{x}$, with $W_J \un{x}^{\un{f}} = W_J z$. Thus there is some map $SLL_{\un{x}, \un{f}} \in \mathcal{S}_{\un{x}, \un{z}}$. At the $(M - 1)$'th step in the construction of this $SLL_{\un{x}, \un{f}}$, we have a map $SLL_{M - 1} : \un{x}_{\leq M - 1} \rightarrow \un{z}'_{M - 1}$, where $\un{z}'_{M - 1}$ is a rex of $z_{M - 1}$. However, $\un{z}'_{M - 1}$ may not be the same as $\un{z}_{M - 1}$. Thus, to construct $SLL_{\un{x}_{\leq M - 1}, \un{e}'} : \un{x}_{\leq M - 1} \rightarrow \un{z}_{M - 1}$, we take $SLL_{M - 1}: \un{x}_{\leq M - 1} \rightarrow \un{z}'_{M - 1}$ and then compose with some rex move $\beta : \un{z}'_{M - 1} \rightarrow \un{z}_{M - 1}$. 

So now, we take the map $\widetilde{SLL}_{M - 1}$ in the above picture and replace it with these maps $SLL_{\un{x}_{\leq M - 1}, \un{e}}$ which we have constructed, plus maps in $I_{\un{z}_{M - 1}}$.
\begin{equation} \label{18}
    \scalebox{0.8}{\input{fig18.tikz}}
\end{equation}
The map on the right of (\ref{18}) represents some map in $I_{\un{z}_{M - 1}}$, followed by $\phi_M$. In particular, the down-dot may be at any position above the box labelled `??'. 

Now, we want to show that this sum is a sum of maps in $\mathcal{S}_{\un{x}, \un{z}}$ with polynomials in the top-left, plus maps in $I_{\un{z}}$. Let us first look at one of the maps in the first sum in (\ref{18}). By the way we have constructed $SLL_{\un{x}_{\leq M - 1}, \un{e}'}$, this will be of the form
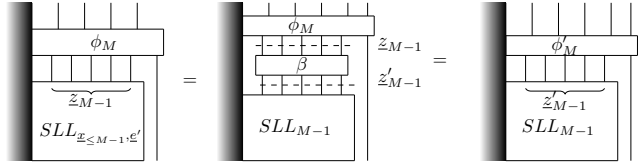
\begin{center}
    \scalebox{0.7}{\input{fig19.tikz}}
\end{center}
where we have composed $\beta$ and $\phi_M$ to make $\phi'_M$, which is another valid construction of this step in the light-leaf construction. Again, $SLL_{M - 1}$ is the $(M - 1)^{\text{th}}$ step in the construction of $SLL_{\un{x}, \un{f}} \in \mathcal{S}_{\un{x}, \un{z}}$. So the map on the right is a valid construction of $SLL_{\un{x}, \un{f}}$, but may not be precisely the map $SLL_{\un{x}, \un{f}} \in \mathcal{S}_{\un{x}, \un{z}}$. In particular, the final step $\phi'_M$ may contain different rex moves. Call this final map $\phi''_M$ in $SLL_{\un{x}, \un{f}}$. Using Lemma \ref{rex}, we have 
\begin{equation} \label{20}
    \scalebox{0.7}{\input{fig20.tikz}}
\end{equation}
where the map labelled `??' has a strictly negative-positive decomposition. In particular, `??' has a down-dot at the top. If $d'_M \neq X1$, then both the maps on the right of (\ref{20}) are of the form we want: the map $\phi''_M$ does not plug into the wall, so we can move the polynomial up past $\phi''_M$ to the top-left corner, leaving us with exactly $SLL_{\un{x}, \un{f}}$ with polynomials in the top-left. And if $d'_M \neq X1$, then the down-dot on top of `??' means that map is in $I_{\un{z}}$.

So the only difficult case is where $d'_M = X1$. In this case we have, using polynomial forcing
\begin{equation} \label{21}
    \scalebox{0.7}{\input{fig21.tikz}}
\end{equation}
where $\gamma$ is some rex move. The first map on the right of (\ref{21}) is again exactly $SLL_{\un{x}, \un{f}}$ with polynomials in the top left. As for the map on the far right of (\ref{21}), the up-dot pulls through $\gamma$ using the Jones-Wenzl relation, leaving us with maps of one of the two following forms:
\begin{equation} \label{22}
    \scalebox{0.7}{\input{fig22.tikz}}
\end{equation}
Now we use an argument that we will use again later, so we refer to it as the argument $(\ast)$. Let $\un{z}'_{M - 1} = (s_1, ..., s_k)$. In the first case in (\ref{22}), the map `??' is a map $(s_1, ..., \hat{s_i}, ..., s_k, t_M) \rightarrow \un{z}$. By Proposition \ref{nll}, this map `??' is spanned by light-leaves and maps in $I_{\un{z}}$. Light-leaves are negative or neutral maps, and since the input and output of `??' have the same length, the light-leaves must be neutral. So if there is a light-leaf $LL : (s_1, ..., \hat{s_i}, ..., s_k, t_M) \rightarrow \un{z}$, then it is a rex move and thus $(s_1, ..., \hat{s_i}, ..., s_k, t_M)$ is reduced, and $s_1 \cdots \hat{s_i} \cdots s_k t_M = z$. But $d'_M = X1$, so that by Lemma \ref{wall-crossing} (a), $z t_M > z$, whereas $s_1 \cdots \hat{s_i} \cdots s_k t_M t_M = s_1 \cdots \hat{s_i} \cdots s_k < s_1 \cdots \hat{s_i} \cdots s_k t_M$. So we have a contradiction, and thus there can be no light-leaf map $LL : (s_1, ..., \hat{s_i}, ..., s_k, t_M) \rightarrow \un{z}$, so that the first map in (\ref{22}) lies entirely in $I_{\un{z}}$. (There ends argument ($\ast$)). 

Then we consider the second map in (\ref{22}). Once again, the box labelled `??' is a sum of light-leaves and maps in $I_{\un{z}}$. Since $\un{z}'_{M - 1}$ and $\un{z}$ are of the same length, a light-leaf map $LL: \un{z}'_{M - 1} \rightarrow \un{z}$ must be a neutral map, i.e. a rex move. But placing a rex move in the box labelled `??' would leave us with an SLL map with $X0$ at the end; more precisely, it would leave us with a possible construction of $SLL_{\un{x}, \un{f}'}$, where $\un{f}' := (\un{e}', 0) \subset \un{x}$. Unfortunately, it may not leave us with \textit{the} map $SLL_{\un{x}, \un{f}'} \in \mathcal{S}_{\un{x}, \un{z}}$. This is because $SLL_{M - 1}$ was chosen to be part of $SLL_{\un{x}, \un{f}} \in \mathcal{S}_{\un{x}, \un{z}}$, not $SLL_{\un{x}, \un{f}'}$. Thus, we need to reapply our inductive hypothesis to this $SLL_{M - 1}$; that is, we replace $SLL_{M - 1}$ with a sum of light-leaves and maps in $I_{\un{z}'_{M - 1}}$ just as in equation (\ref{18}), and then proceed through the steps we have just carried out (or are about to carry out). Note that there is no danger here of an infinite loop: we only reached the situation in (\ref{22}) because we had $d_k' = X1$. If we repeat the procedure with $d_k' = X0$, we will not run into the same problem. 

We have yet to address the right-hand map in (\ref{20}), in the case that $d'_M = X1$:
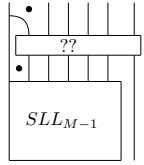
\begin{equation}
    \scalebox{0.7}{\input{fig23.tikz}}
\end{equation}
Here, again, the box labelled `??' has a strictly negative-positive decomposition, so that it has a dot at the top and at the bottom. If the dot at the top is not connected to the wall strand, then this map is in $I_{\un{z}}$. Otherwise, the dot gets pulled into the wall and we have a map that looks like one of the following two:
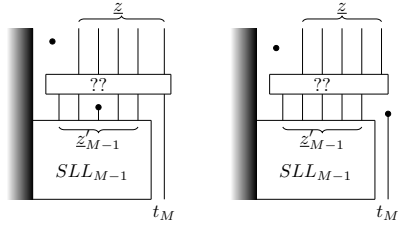
\begin{equation}
    \scalebox{0.7}{\input{fig22.tikz}}
\end{equation}
This is exactly the situation in (\ref{22}), which we have already dealt with. 

We have now dealt with maps in the first sum in (\ref{18}). We now move on to maps in the second sum, of the form
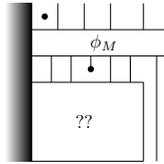
\begin{equation}
    \scalebox{0.7}{\input{fig24.tikz}}
\end{equation}
where again, the down-dot could be positioned anywhere above `??'. First, if $d'_M$ is either $U0, U1$ or $X0$, then the down-dot will pull through any rex moves to the top, landing in $I_{\un{z}}$. 

If $d'_M = D0$ or $D1$, we can assume there are no `necessary' rex moves (labelled $\beta$ in Section \ref{spherical light-leaves}) as part of $\phi_M$, since if there were, we could pull the down-dot through $\beta$ and then incorporate what remained into `??'. Therefore, if the down-dot is not on the very right of `??', then it will again pull upward through any rex moves to land the map in $I_{\un{z}}$. If the down-dot is on the very right of `??', then we are in one of the following situations:
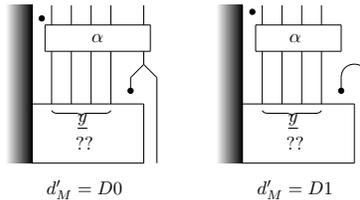
\begin{equation}
    \scalebox{0.7}{\input{fig25.tikz}}
\end{equation}
where $\alpha$ is some rex move, and $\un{y}$ is a rex of some $y \in \jw$ (because $d_k = D0$ or $D1$). Therefore we can reapply our inductive hypothesis to the box labelled `??', replacing it with SLL's of our choice (with polynomials in the top-left), and maps in $I_{\un{y}}$. These SLL maps will become (up to adjusting the final rex move $\alpha$, as in (\ref{20})) the SLL maps in $\mathcal{S}_{\un{x}, \un{z}}$ corresponding to $d'_M = U1$ and $U0$ respectively. The dots at the top of the maps in $I_{\un{y}}$ will pull through $\alpha$ to land in $I_{\un{z}}$. 

So we are left with the case $d'_M = X1$. 
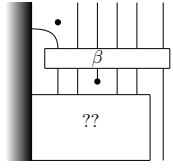
\begin{equation}
    \scalebox{0.7}{\input{fig26.tikz}}
\end{equation}
The down-dot will pull through $\beta$, either landing in $I_{\un{z}}$ or pulling into the wall. If it pulls into the wall, then by the argument $(\ast)$, what remains will be in $I_{\un{z}}$. $\hfill \square$\\
\\
\noindent Thus, we have shown that any map $\phi: \un{x} \rightarrow \un{z}$ is a sum of any particular choice of SLL's, with polynomials in the top-left. The only thing left to do is to move these polynomials to the right, using polynomial forcing. This leaves a right $R$-linear combination of our choice of SLL's, plus maps in $I_{\un{z}}$. This proves Proposition \ref{ll}, and thus also Proposition \ref{span}. Therefore, we have proved that the double-leaves span, and also that they are linearly independent (Corollary \ref{linearly independent}). So we have proved Theorem \ref{basis}, that the double-leaves form a basis.  $\hfill \square$

\section{Consequences of the Main Theorem} \label{Consequences}
In this chapter we give three consequences of the main theorem, Theorem \ref{basis}, that double-leaves form a basis of $\mathcal{M}_{BS}$. As with all other chapters in this paper, we require the conditions on our realization given in Section \ref{realizations}.

\subsection{Equivalence of Diagrammatic and Algebraic Categories} \label{Equivalence of Diagrammatic and Algebraic Categories}

\begin{thm} \label{equivalence}
    Suppose our realization $\mathfrak{h}$ is reflection faithful. Then the functor $\tA : \mathcal{M}_{BS} \rightarrow J\mathbb{BS}$Bim defined in Section \ref{The diagrammatic category} is an equivalence of categories.   
\end{thm}

The key to the proof will be that we have formulae for the graded ranks of the Hom spaces in each category, and that these are the same. From Theorem \ref{Singular thm} (c), we know that the following diagram commutes. 
\begin{center}
    \input{fig117.tikz}
\end{center} By following where $_{R^J} R\otimes BS(\un{w}) \in \prescript{J}{}{\mathcal{B}} \times \mathcal{B}$ is taken in this diagram, we can determine that $$ch(_{R^J} BS(\un{w})) = m_{id} b_{\un{w}} = 1 \otimes b_{\un{w}} \in M(J).$$ 

We also have a $\mathbb{Z}[v, v^{-1}]$-bilinear form $\lgl \cdot, \cdot\rgl_{I, J}$ on $\isj$ for all finitary subsets $I, J \subset S$. We will not define this form on general $\isj$, since this would require defining the standard basis of each $\isj$. We will simply define it on $M(J) = \prescript{J}{}{\mathcal{S}}$. We define a $\mathbb{Z}[v, v^{-1}]$-bilinear form $\lgl \cdot, \cdot \rgl_{M}$ on $M(J)$ by identifying $M(J) \cong b_{w_J} H$ as in Proposition \ref{smodule} and setting\footnote{We have added in a factor of $v^{-d_J}$ to the formula given in Section 2.3 of \cite{Singular}. This is in order to make the standard basis of $M$ orthonormal with respect to this form. We have altered the Hom formula appropriately to adjust for this.} \begin{equation} \label{Mform}
\begin{split}
    \lgl \cdot, \cdot \rgl_M : b_{w_J}H \times b_{w_J}H \rightarrow \mathbb{Z}[v, v^{-1}],\\
    \lgl h_1, h_2 \rgl_M := v^{-d_J}\epsilon(i(h_1) \ast_J h_2) = \frac{v^{-d_J}}{\pi(J)}\lgl h_1, h_2 \rgl.
    \end{split}
\end{equation}

\noindent Finally, we have a formula for the graded rank of $$\text{Hom}^\bullet_{J\mathbb{BS}\text{Bim}}(_{R^J} BS(\un{w}), _{R^J} BS(\un{y}))$$ as a right $R$-module. 
\begin{prop} \label{JHom}
    Suppose our realization is reflection-faithful. Then we have \begin{align*}\ol{\text{rkHom}^\bullet_{J\mathbb{BS}\text{Bim}}(_{R^J} BS(\un{w}), \; _{R^J} BS(\un{y}))} & = \lgl ch(BS(\un{w})), ch(BS(\un{y}))\rgl_M\\ & = \lgl 1 \otimes b_{\un{w}}, 1 \otimes b_{\un{y}}\rgl_M.
    \end{align*}
\end{prop}
\noindent \textit{Proof.} This is a specific case of Theorem 7.9 of \cite{Singular}\footnote{Again, we have removed the grading shift that shows up in \cite{Singular} to match the extra factor of $v^{-d_J}$ in (\ref{Mform})}.  $\hfill \square$\\
\\
\noindent We now want to show that the graded ranks of the Hom spaces in $J\mathbb{BS}$Bim match those of $\mathcal{M}_{BS}$. 

\begin{prop} \label{matching ranks}
    Suppose our realization is reflection-faithful. Then given any two expressions $\un{x}, \un{y}$, we have $$\text{rkHom}^\bullet_{\mathcal{M}_{BS}}(\un{x}, \un{y}) = \text{rkHom}_{J\mathbb{BS}\text{Bim}}(\prescript{}{R^J}{BS(\un{x})},  \prescript{}{R^J}{BS(\un{y})}).$$
\end{prop} There is a notion of \textit{defect} $sdef(\un{e})$ of subexpressions $\un{e} \subset \un{x}$, (see Section 3.3.4 of \cite{textbook}) which matches the degree of the corresponding light-leaf $LL_{\un{x}, \un{e}}$. We will now define a new notion of defect, the \textit{spherical defect} $sdef(\un{e})$ of a subexpression $\un{e} \subset \un{x}$, in such a way that deg$(SLL_{\un{x}, \un{e}}) = sdef(\un{e})$. 

Recall from Section \ref{spherical light-leaves} that given an expression $\un{x} = (s_1, ..., s_n)$, a subexpression $\un{e} \subset \un{x}$ has a coset stroll $z_0, ..., z_n$, where $z_k$ is the minimal coset representative of $x_k := s_1^{e_1} \cdots s_k^{e_k}$. This coset stroll has an associated decoration $d_1', ..., d_n'$, with $d_k' \in \{U1, U0, D1, D0, X1, X0\}$. These $d_k'$'s then determined the form of the spherical light-leaf $SLL_{\un{x}, \un{e}}$, by specifying the maps $\phi_k$. If we want to know how to define $sdef(\un{e})$, we should look at the degrees of the $\phi_k$'s corresponding to each of the possible labels $U1, U0, D1, D0, X1, X0$. The degree of $\phi_k$ associated to $U1$ and $D1$ is zero, while $U0$ is +1 and $D0$ is -1. Then, $X0$ gives a degree +1 map, while $X1$ gives a degree -1 map. Therefore we should define the spherical defect as \begin{equation} \label{spherical defect}
    sdef(\un{e}) := \#U0 + \#X0 - \#D0 - \#X1.
\end{equation} By the way we have defined it, we have deg$(SLL_{\un{x}, \un{e}}) = sdef(\un{e})$. By the fact that double-leaves form a basis, we have the following Hom formula for $\mathcal{M}_{BS}$. \begin{equation} \label{MHom}
    \ol{\text{rkHom}^\bullet_{\mathcal{M_{BS}}}(\un{x}, \un{y})} = \sum_{\substack{\un{e} \subset \un{x}\\ \un{f} \subset \un{y}\\ W_J \un{x}^{\un{e}} = W_J \un{y}^{\un{f}}}}v^{sdef(\un{e}) + sdef(\un{f})}
\end{equation}
We now wish to show that we can rewrite the Hom formula for $J\mathbb{BS}$Bim from Proposition \ref{JHom} in terms of the spherical defect, so that it matches the above formula (\ref{MHom}).

Recall that we have a standard basis $\{m_x \; | \; x \in \jw \}$ of the spherical module $M$, where $m_x = 1 \otimes \dl_x$.
\begin{lem} \label{Morth}
    The standard basis $\{m_x \; | \; x \in \jw\}$ of $M(J)$ is orthonormal with respect to the form $\lgl \cdot, \cdot \rgl_M$.
\end{lem}
\noindent \textit{Proof.} Under the identification $M(J) \cong b_{w_J} H$, the standard basis element $m_x \in M(J)$ goes to $b_{w_J} \dl_x \in b_{w_J} H$. Therefore by (\ref{Mform}) we have $$\lgl m_x, m_y\rgl_M = \frac{v^{-d_J}}{\pi(J)}\lgl b_{w_J} \dl_x, b_{w_J} \dl_y\rgl.$$
It is a straight-forward consequence of the definition (\ref{bilinear form}) of $\lgl \cdot, \cdot \rgl$ that \begin{align*}\lgl b_{w_J} \dl_x, b_{w_J} \dl_y\rgl & = \lgl i(b_{w_J})b_{w_J} \dl_x, \dl_y\rgl\\
& = \lgl b_{w_J}^2 \dl_x, \dl_y\rgl\\
& = \pi(J)\lgl b_{w_J} \dl_x, \dl_y \rgl,
\end{align*} where $i(b_{w_J}) = b_{w_J}$ by Remark \ref{anti remark}. This gives us $$\lgl m_x, m_y\rgl_M = v^{-d_J} \lgl b_{w_J} \dl_x, \dl_y \rgl.$$ We have from (\ref{bw0}) that $$b_{w_J} = \sum_{w \in W_J} v^{l(w_J) - l(w)} \dl_w,$$ and so $$b_{w_J}\dl_x = \sum_{w \in W_J} v^{l(w_J) - l(w)} \dl_{wx},$$ since $x$ is a minimal coset representative. Therefore if $W_J x \neq W_J y$, then $wx \neq y$ for any $w \in W_J$, so that $$\lgl m_x, m_y\rgl_M = v^{-d_J} \lgl b_{w_J} \dl_x, \dl_y \rgl = 0,$$ by the fact that the standard basis of $H$ is orthonormal with respect to $\lgl \cdot, \cdot \rgl$.

Conversely if $W_J x= W_J y$, then $x = y$ since they are both minimal coset representatives. Therefore $$\lgl b_{w_J} \dl_x, \dl_y\rgl = \lgl \sum_{w \in W_J} v^{l(w_J) - l(w)} \dl_{wy}, \dl_y\rgl = v^{l(w_J)} = v^{d_J}.$$ Hence $$\lgl m_x, m_y\rgl_M = v^{-d_J} \cdot v^{d_J} = 1,$$ and so the $m_x$'s are orthonormal. $\hfill \square$\\

\noindent We will use the notation $m_{ \un{x}^{\un{e}}} := m_z$, where $z$ is the minimal coset representative of $W_J \un{x}^{\un{e}}$. 

\begin{lem} \label{1 bx}
    For any expression $\un{x}$ we have \begin{equation}
        1 \otimes b_{\un{x}} = \sum_{\un{e} \subset \un{x}} v^{sdef({\un{e}})} m_{\un{x}^{\un{e}}}.
    \end{equation}
\end{lem}
\noindent In other words, the coefficient of $m_z$ in $1 \otimes b_{\un{x}}$ is $$ \sum_{\substack{\un{e} \subset \un{x}\\W_J \un{x}^{\un{e}} = W_J z}} v^{sdef(\un{e})}.$$
\noindent \textit{Proof.} We proceed by induction on the length of $\un{x}$. Suppose the statement is true when $l(\un{x}) = n - 1$, and take $\un{x} = (s_1, ..., s_n)$. Set $\un{y} := (s_1, ..., s_{n-1})$.  Then we know by the inductive hypothesis that $$1 \otimes b_{\un{y}} = \sum_{\un{f} \subset \un{y}} v^{sdef(\un{f})} m_{\un{y}^{\un{f}}},$$ so that \begin{equation} \label{lemma 1.4 1}
    1 \otimes b_{\un{x}} = 1 \otimes b_{\un{y}} b_{s_n} = \sum_{\un{f} \subset \un{y}} v^{sdef(\un{f})} m_{\un{y}^{\un{f}}} b_{s_n}.
\end{equation} We can find the product $m_{\un{y}^{\un{f}}} b_{s_n}$ using the formula (\ref{smult}), which we reproduce below.
\begin{equation} \label{smult 2}
m_w b_s = \begin{cases}
    m_{ws} + v^{-1} m_w & \text{if } ws < w, \; ws \in \jw\\
    m_{ws} + vm_{w} & \text{if } ws > w, \; ws \in \jw\\
    (v+v^{-1})m_{w} & \text{if } ws \notin \jw.
\end{cases} \end{equation} Thus each term $m_{\un{y}^{\un{f}}} b_{s_n}$ in (\ref{lemma 1.4 1}) will yield two terms. These two terms will be $m_{\un{x}^{\un{e}}}$ for $\un{e} = (\un{f}, 0)$ and $(\un{f}, 1)$, with some powers of $v$ in front. We just need to show that these powers of $v$ are exactly $sdef(\un{e})$.

Note that the three cases in (\ref{smult 2}) correspond exactly to the cases where $s$ is labelled by $D, U$ or $X$ in the labelling of the coset stroll (see Section \ref{spherical light-leaves}). More precisely, if we replace $w$ with $\un{y}^{\un{f}}$ and $s$ with $s_n$ in the above formula, then the three cases correspond to when the labellings of the coset strolls associated to $(\un{f}, 0)$ and $(\un{f}, 1)$ finish with $D, \; U$ or $X$ respectively. When they finish with $D$, then using the definition of $sdef$ in (\ref{spherical defect}) we will have $$sdef(\un{f}, 0) = sdef(\un{f}) - 1, \; sdef(\un{f}, 1) = sdef(\un{f}),$$ which exactly corresponds to the powers of $v$ in the first case of (\ref{smult 2}). If the expressions finish with $U$, then $$sdef(\un{f}, 0) = sdef(\un{f}) + 1, \; sdef(\un{f}, 1) = sdef(\un{f}),$$ which corresponds to the powers of $v$ in the second case of (\ref{smult 2}). Finally, if the labelling of the coset strolls of $(\un{f}, 0), \; (\un{f}, 1) \subset \un{x}$ finish with $X$, then we will have $$sdef(\un{f}, 0) = sdef(\un{f}) + 1, \; sdef(\un{f}, 1) = sdef(\un{f}) - 1, $$ which corresponds to the final case in (\ref{smult 2}).

Therefore the coefficient of $m_{\un{x}^{\un{e}}}$ in each case will be $v^{sdef(\un{e})}$. $\hfill \square$\\
\\
\noindent \textit{Proof of Proposition \ref{matching ranks}.} Proposition \ref{JHom} tells us that $$\ol{\text{rkHom}^\bullet_{J\mathbb{BS}\text{Bim}}(_{R^J} BS(\un{w}), _{R^J} BS(\un{y}))} = \lgl 1 \otimes b_{\un{w}}, 1 \otimes b_{\un{y}}\rgl_M.$$ Lemmas \ref{Morth} and \ref{1 bx} tell us that \begin{align*}\lgl 1 \otimes b_{\un{x}}, 1 \otimes b_{\un{y}}\rgl_M & = \lgl \sum_{\un{e} \subset \un{x}} v^{sdef(\un{e})} m_{\un{x}^{\un{e}}}, \sum_{\un{f} \subset \un{y}} v^{sdef(\un{f})} m_{\un{y}^{\un{f}}}\rgl_M\\
& = \sum_{\substack{\un{e} \subset \un{x}\\ \un{f} \subset \un{y}\\ W_J \un{x}^{\un{e}} = W_J \un{y}^{\un{f}}}}v^{sdef(\un{e}) + sdef(\un{f})},
\end{align*} which is exactly $\ol{\text{rkHom}^\bullet_{\mathcal{M_{BS}}}(\un{x}, \un{y})}$ given in (\ref{MHom}). $\hfill \square$\\
\\
\noindent We now know that $\mathcal{M}_{BS}$ and $J\mathbb{BS}$Bim have matching Hom space ranks. We will use this fact to show that $\tA: \mathcal{M}_{BS} \rightarrow J\mathbb{BS}$Bim is an equivalence of categories. \\
\\
\noindent \textit{Proof of Theorem \ref{equivalence}.} It is clear that $\tA$ is essentially surjective, so we only need to show it is fully faithful.

We will first show that it is faithful. Consider the following diagram (see Chapter \ref{Standard Diagrammatics}).
\begin{equation} 
    \input{fig118.tikz}
\end{equation}
We can extend this diagram on the left.
\begin{equation} \label{com3}
    \input{fig119.tikz}
\end{equation}
When we proved linear independence of the spherical double-leaves in Section \ref{Section: Proof of linear independence}, we in fact proved that they are linearly independent after mapping from $\mathcal{M}_{BS}$ to $Kar(\mathcal{M}_{BS, Q})$ across the top of (\ref{com3}). This means that the map $\mathcal{M}_{BS} \rightarrow Kar(\mathcal{M}_{BS, Q})$ across the top of (\ref{com3}) is faithful. Furthermore, since by Corollary \ref{Gequivalence} the functor $\tilde{\mathcal{G}}$ is an equivalence\footnote{Reflection-faithful implies faithful.}, we find that the map $\mathcal{M}_{BS} \rightarrow JStd$Bim$_Q$ in (\ref{com3}) is faithful. Therefore the map $\tilde{\mathcal{A}}$ must also be faithful.

Finally, since by Proposition \ref{matching ranks} the graded ranks of the Hom spaces of $\mathcal{M}_{BS}$ and $J\mathbb{BS}$Bim are the same, the fact that $\tA$ is faithful implies that it is full. $\hfill \square$
\begin{cor}
    Suppose our realization is reflection-faithful. Then the diagrammatic spherical category $\mathcal{M} := Kar(\mathcal{M}_{BS})$ is equivalent to the algebraic spherical category $Kar(J\mathbb{BS}\text{Bim})$.
\end{cor}

So far we have only shown that the diagrammatic and algebraic spherical categories, $\cM$ and $Kar(\jbs)$, are equivalent as categories, with no additional structure. However, they are also equivalent as module categories.

\begin{cor} \label{main theorem}
    Suppose our realization is reflection-faithful. Then the diagrammatic and algebraic spherical categories, $\cM$ and $Kar(\jbs)$, are equivalent as module categories over $Kar(\mathcal{D})$ and $\mathbb{S}$Bim respectively. That is, the following diagram commutes.
    
\begin{center}
    \input{fig134.tikz}
\end{center}
\end{cor}

\noindent \textit{Proof.} This is a direct consequence of the fact that the corresponding Bott-Samelson categories are equivalent as module categories. That is, the following diagram commutes.
\begin{center}
    \input{fig135.tikz}
\end{center}
This in turn is a consequence of how the functor $\tilde{\mathcal{A}}$ was defined in Section \ref{The diagrammatic category}. In particular, $\tilde{\mathcal{A}}$ clearly respects the right action of $\mathcal{D}$ on $\mathcal{M}_{BS}$, both on the level of objects and morphisms. Taking Karoubi envelopes of this commuting diagram proves the result. $\hfill \square$

\subsection{Categorification of the Spherical Module}

In this section we will prove that the diagrammatic category $\mathcal{M}$ is indeed a spherical category; that is, that it categorifies the spherical module. The proof follows almost exactly the one at the end of Chapter 6 of \cite{SoergelCalculus}. The crucial facts we will use are that the spherical double leaves form a basis, and that we are assuming $\mathbbm{k}$ is a complete local ring. 

\begin{prop} \label{SSCT}
    For each $w \in {}^JW$, there is a unique summand $C_w$ of $\un{w} \in \mathcal{M}$ (for $\un{w}$ reduced) which is not isomorphic to a shift of any summand of $\un{v}$, for $\un{v}$ a reduced expression for an m.c.r. $v < w$. The object $C_w$ does not depend on the reduced expression $\un{w}$ up to isomorphism. Moreover any indecomposable object in $\mathcal{M}$ is isomorphic to a shift of $C_w$ for some $w \in {}^J W$. Hence, one has a bijection:
    \begin{align*}
{}^JW &\xrightarrow{\sim} \left\{ \begin{gathered} \textit{indecomposable objects in } \mathcal{M} \\ \textit{up to shifts and isomorphism} \end{gathered} \right\} \\
w &\mapsto C_w.
\end{align*}
\end{prop}
\begin{proof}
    Since $\mathbbm{k}$ is a complete local ring, the category $\mathcal{M}$ is Krull-Schmidt (see \cite[Lem.~6.25]{SoergelCalculus}). The proof then follows exactly as that of \cite[Thm.~6.26]{SoergelCalculus}.
\end{proof}
Let us now define a diagrammatic character. Given $w \in {}^J W$, let $\mathcal{M}^{\geq w}$ be the quotient category of $\mathcal{M}$ where any morphism factoring through $\un{y}$ a reduced expression for an m.c.r. $y \ngeq w$ is zero. That is, since double leaves span morphism spaces, any diagram factors through $\un{y}$ for some $y \in {}^J W$, and in $\mathcal{M}^{\geq w}$ we ignore all morphisms which factor through elements $y \ngeq w$ (see Section 6.5 of \cite{SoergelCalculus} for more details). 

Now, let $[\mathcal{M}]$ denote the split Grothendieck group of $\mathcal{M}$. This is a right module over $[Kar(\mathcal{D})] \cong H$.

\begin{defn}
    We define the \textit{diagrammatic character} by 
    \begin{align*}
ch : [\mathcal{M}] &\to M \\
[C] &\mapsto \sum_{w \in {}^JW} \overline{\operatorname{rk} \operatorname{Hom}_{\mathcal{M}^{\ge w}}(C, \underline{w})} m_w.
\end{align*}
\end{defn}
From the fact that double leaves form a basis, one can see that $\operatorname{Hom}_{\mathcal{M}^{\geq w}}(\un{x}, \un{w})$ is a free $R$-module spanned by light-leaves $SLL_{\un{x}, \un{e}}$ for $\un{e}$ expressing $\un{w}$. Thus by Lemma \ref{1 bx}, we have $ch([\un{x}]) = 1 \otimes b_{\un{x}}$. Moreover, abusing notation, we have a diagrammatic character $ch: [Kar(\mathcal{D})] \rightarrow H$ defined completely analogously in \cite[Defn.~6.24]{SoergelCalculus}, and for $\un{x} \in \mathcal{M}, \; \un{y} \in Kar(\mathcal{D})$, we have $ch(\un{x} \cdot \un{y}) = 1 \otimes b_{(\un{x}, \un{y})} = 1 \otimes b_{\un{x}} b_{\un{y}} = ch(\un{x}) ch(\un{y})$, so that $ch$ is a morphism of $H$-modules. 

\begin{prop}
    The diagrammatic character \[ch: [\mathcal{M}] \rightarrow M\] is an isomorphism of $H$-modules. 
\end{prop}
\begin{proof}
    By Proposition \ref{SSCT}, the objects $[C_x]$ form a $\mathbb{Z}[v, v^{-1}]$-basis of $[\mathcal{M}]$. Then we have \[ch([C_x]) = \sum_{w\leq x} g_{w, x} m_w\] for some $g_{w, x} \in \mathbb{Z}[v, v^{-1}]$, with $g_{x, x} = 1$. Therefore the images of $[C_x]$ are upper-triangular with respect to the Bruhat order, and so form a basis. Thus $ch$ maps a basis to a basis, and is therefore an isomorphism. 
\end{proof}

\subsection{Equivalence to Abe's Soergel bimodules}

We have shown that the diagrammatic and algebraic spherical categories are equivalent when our realization is reflection-faithful. Abe has constructed an algebraic spherical category $_J\mathcal{S}$bimod (which he calls one-sided Soergel bimodules) which behaves well even without the assumption of reflection-faithfulness. Throughout this section we impose on our realization the conditions given at the beginning of Section 2.3 of \cite{Abe2023} (as well as those in Section \ref{realizations} of this paper).

\begin{defn}
    Given finitary $J \subset S$, let ${}_J\mathcal{C}$ be defined as follows. An object of ${}_J \mathcal{C}$ is a graded $(R^J, R)$-bimodule $M$ together with a decomposition \[ M \otimes_R Q \cong \bigoplus_{x \in {}^J W} M_Q^x\] upon localization, where $m \cdot r = x(r)m$ for $m \in M_Q^x$. A morphism $\varphi :M \rightarrow N$ in ${}_J \mathcal{C}$ is a homogeneous morphism such that $\varphi(M_Q^x) \subset \varphi(N_Q^x)$ for all $x \in {}^J W$. We let the \textit{support} supp$(M)$ denote the set of $x \in {}^J W$ such that $M_Q^x \neq 0$.
\end{defn}

Recall from Lemma \ref{hom2} and Remark \ref{Q index} that when our realization is faithful, objects in $J \mathbb{BS}\text{Bim}$ satisfy the above decomposition. However, when the realization is not faithful, we may have non-zero morphisms $ \varphi :Q_x \rightarrow Q_y$ when $W_J x \neq W_J y$, and therefore objects of $J \mathbb{BS} \text{Bim}$ do not necessarily satisfy the conditions of the above definition. 
\begin{defn}
    The $R$-bimodule $B_s := R \otimes_{R^s} R(1)$ has a unique structure as an object of $\mathcal{C} := {}_{\varnothing}\mathcal{C}$ with support $\{id, s\}$. We let $\mathcal{BS}bimod$ denote the smallest full subcategory of $\mathcal{C}$ containing $R$, and $B_s$ for all $s \in S$, which is closed under $\oplus, \otimes, (1), (-1)$. 
\end{defn}

\begin{defn}
    Given a finitary subset $J \subset S$, $_J\mathcal{BS}$bimod consists of objects $_{R^J}B$ for $B \in \mathcal{BS}$bimod. The category $_J\mathcal{S}$bimod is the Karoubi envelope of $_J\mathcal{S}$bimod.
\end{defn}
Let $U:{}_{J}\mathcal{BS}\mathrm{bimod}\to J\mathbb{BS}\mathrm{Bim}$ denote the forgetful functor.
\begin{prop}\label{prop:AbeComparisonFunctor}
There exists a functor
\[
\mathcal{F}:\mathcal{M}_{BS}\longrightarrow {}_{J}\mathcal{BS}\mathrm{bimod}
\]
such that the diagram
\begin{equation}\label{eq:AbeCommutativeTriangle}
\begin{tikzcd}[row sep=large, column sep=large]
    & \mathcal{M}_{BS} \arrow[ld, "\mathcal{F}"'] \arrow[rd, "\tilde{\mathcal{A}}"] & \\
    {}_{J} \mathcal{BS}\mathrm{bimod} \arrow[rr, "U"'] & & J\mathbb{BS}\mathrm{Bim}
\end{tikzcd}
\end{equation}
commutes.
\end{prop}
\begin{proof}
    On objects, we define $\mathcal{F}(\un{w}) := {}_{R^J} BS(\un{w})$. On morphisms, $\mathcal{F}$ is defined such that the diagram (\ref{eq:AbeCommutativeTriangle}) commutes. By Corollary \ref{HomStdQ} and Proposition \ref{Std}, this is well-defined.
\end{proof}

\begin{prop}
    The functor $\mathcal{F} :\mathcal{M}_{BS} \rightarrow {}_J\mathcal{BS}\operatorname{bimod}$ is an equivalence of categories.
\end{prop}
\begin{proof}
    It is clear that $\mathcal{F}$ is essentially surjective. By Corollary \ref{Gequivalence}, and the diagram (\ref{com3}), we have that $\tilde{\mathcal{A}}$ is faithful. Therefore by (\ref{eq:AbeCommutativeTriangle}), the functor $\mathcal{F}$ is faithful. Moreover, by the Hom formula in \cite[Thm. ~2.38]{Abe2023}, the graded ranks of Hom spaces are the same. Therefore $\mathcal{F}$ is an equivalence.
\end{proof}

\begin{cor}
    The functor $\mathcal{F} : \mathcal{M} \rightarrow {}_J \mathcal{S}\text{bimod}$ is an equivalence of module categories over $Kar(\mathcal{D}) \cong \mathcal{S}bimod$.
\end{cor}

\bibliographystyle{alpha}
\bibliography{bibliography-2}
\end{document}

%% file: sample.tikzstyles
\tikzstyle{small dot}=[fill=black, draw=black, shape=circle, scale=0.3]
\tikzstyle{Dot}=[fill=black, draw=black, shape=circle, scale=0.6]
\tikzstyle{red dot}=[fill=red, draw=red, shape=circle, scale=0.5]
\tikzstyle{blue dot}=[fill={rgb,255: red,0; green,17; blue,255}, draw={rgb,255: red,0; green,17; blue,255}, shape=circle, scale=0.5]
\tikzstyle{Box?}=[shape=rectangle, draw=black, inner sep=0.05 cm]

\tikzstyle{Up Brace}=[-, tikzit draw={rgb,255: red,229; green,6; blue,36}, decoration={{brace,amplitude=3pt, mirror}}, decorate]
\tikzstyle{Down Brace}=[-, decoration={{brace, amplitude=3pt}}, decorate, tikzit draw={rgb,255: red,0; green,4; blue,255}]
\tikzstyle{Dashed Green}=[-, draw={rgb,255: red,59; green,201; blue,34}, dashed]
\tikzstyle{Dashed}=[-, dashed]
\tikzstyle{Red}=[-, draw={rgb,255: red,255; green,0; blue,4}]
\tikzstyle{Blue}=[-, draw={rgb,255: red,2; green,18; blue,255}]
\tikzstyle{Green}=[-, draw={rgb,255: red,0; green,208; blue,0}]
\tikzstyle{Grey Dashed}=[-, dashed, draw={rgb,255: red,168; green,168; blue,168}]
\tikzstyle{Arrow}=[-, ->]
\tikzstyle{Mapsto}=[-, {|->}]
\tikzstyle{Red Dashed}=[-, dashed, draw={rgb,255: red,236; green,0; blue,3}]
\tikzstyle{Blue Dashed}=[-, draw={rgb,255: red,20; green,63; blue,252}, dashed]
\tikzstyle{Grey}=[-, draw={rgb,255: red,170; green,170; blue,170}]
\tikzstyle{Hook arrow}=[-, right hook->]